\documentclass[10pt, a4paper]{amsart}

\usepackage{amsthm,mathrsfs,amsmath,amscd,enumerate,enumitem, amscd, xcolor, url}
\usepackage[colorlinks]{hyperref}
\usepackage[bbgreekl]{mathbbol}
\usepackage{amsfonts}
\usepackage{amssymb}
\usepackage[utf8]{inputenc}

\DeclareSymbolFontAlphabet{\amsbb}{bbold}
\DeclareSymbolFontAlphabet{\mathbb}{AMSb}%

\newtheorem{thm}{Theorem}[section]

\newtheorem{lem}[thm]{Lemma}
\newtheorem{prop}[thm]{Proposition}
\theoremstyle{definition}

\theoremstyle{remark}

\DeclareMathOperator{\supp}{supp}
\DeclareMathOperator*{\esssup}{ess\ sup}
\DeclareMathOperator*{\essinf}{ess\ inf}

\newcommand{\zinf}{\mathbb{R}_+}
\newcommand{\zinfn}{\mathbb{R}_+^n}
\newcommand{\zz}{\zinfn\times \zinfn}
\newcommand{\zia}{\zinfn\times(-1,1)^n}
\newcommand{\zzia}{\zinfn\times\zinfn\times(-1,1)^n}
\newcommand{\loc}{\text{loc}}
\newcommand{\glob}{\text{glob}}
\DeclareMathOperator{\Real}{Re}
\newcommand{\mhl}{M_{\textup{HL}}}

\numberwithin{equation}{section}
\allowdisplaybreaks

\begin{document}
    \title[Harmonic analysis operators associated with Laguerre polynomial...]
    {Harmonic analysis operators associated with Laguerre polynomial expansions on variable Lebesgue spaces}

    \dedicatory{Dedicated to the memory of our beloved Eleonor Harboure.}

    \author[J. J. Betancor]{Jorge J. Betancor}
    \address{Jorge J. Betancor\newline
        Departamento de An\'alisis Matem\'atico, Universidad de La Laguna,\newline
        Campus de Anchieta, Avda. Astrof\'isico S\'anchez, s/n,\newline
        38721 La Laguna (Sta. Cruz de Tenerife), Spain}
    \email{jbetanco@ull.es}

    \author[E. Dalmasso]{Estefan\'ia Dalmasso}
    \address{Estefan\'ia Dalmasso, Pablo Quijano\newline
        Instituto de Matem\'atica Aplicada del Litoral, UNL, CONICET, FIQ.\newline Colectora Ruta Nac. Nº 168, Paraje El Pozo,\newline S3007ABA, Santa Fe, Argentina}
    \email{edalmasso@santafe-conicet.gov.ar, pquijano@santafe-conicet.gov.ar}

    \author[P. Quijano]{Pablo Quijano}

    \author[R. Scotto]{Roberto Scotto}
    \address{Roberto Scotto\newline
        Universidad Nacional del Litoral, FIQ.\newline Santiago del Estero 2829,\newline S3000AOM, Santa Fe, Argentina}
    \email{roberto.scotto@gmail.com}

    \thanks{The first author is partially supported by grant PID2019-106093GB-I00 from the Spanish Government. The second and fourth authors are partially supported by grants PICT-2019-2019-00389 (ANPCyT) and CAI+D 2019-015 (UNL)}
    \date{\today}


    %%% ----------------------------------------------------------------------
    \begin{abstract}
        In this paper we give sufficient conditions on a measurable function $p:(0,\infty)^n\rightarrow [1,\infty)$ in order that harmonic analysis operators (maximal operators, Riesz transforms, Littlewood--Paley functions and multipliers) associated with $\alpha$-Laguerre polynomial expansions are bounded on the variable Lebesgue space $L^{p(\cdot)} ((0,\infty)^n, \mu_\alpha)$, where $d\mu_\alpha (x)=2^n\prod_{j=1}^n \frac{x_j^{2\alpha_j+1} e^{-x_j^2}}{\Gamma(\alpha_j+1)} dx$, being $\alpha=(\alpha_1, \dots, \alpha_n)\in [0,\infty)^n$ and $x=(x_1,\dots,x_n)\in (0,\infty)^n$.
    \end{abstract}
    %%% ----------------------------------------------------------------------
    \date{}
    \maketitle
    %%% ----------------------------------------------------------------------

    \section{Introduction and main results}
    In this article we establish $L^{p(\cdot)}-$boundedness properties of harmonic analysis operators appearing in the context of Laguerre polynomials.

    For every $\alpha >-1$ and $k\in \mathbb N:=\{0,1,2,\dots\}$, the normalized Laguerre polynomial of type $\alpha$ and degree $k$ is defined by the formula (c.f. \cite{Leb}, \cite{Sz})
    \[L_k^\alpha(x)= \sqrt{\frac{\Gamma(\alpha+1)}{\Gamma(\alpha+k+1)k!}}e^xx^{-\alpha} \frac{d^k}{dx^k}\left(e^{-x} x^{k+\alpha}\right), \quad x\in (0,\infty).\]
    Let $\alpha=(\alpha_1, \dots, \alpha_n)\in \left(-1,\infty\right)^n$. For every $k=(k_1,\dots, k_n)\in \mathbb N^n$, the $k$--th Laguerre polynomial of type $\alpha$ and degree $\widehat{k}:=k_1+\cdots + k_n$ is defined by
    \[L_k^\alpha(x)=\prod_{i=1}^n L_{k_i}^{\alpha_i}(x_i), \quad x=(x_1,\dots,x_n)\in (0,\infty)^n:=\mathbb{R}_+^n.\]
    The sequence of polynomials $\{L_k^\alpha\}_{k\in \mathbb N^n}$ is an orthonormal basis for $L^2(\zinfn, \nu_\alpha)$ being $d\nu_\alpha (x)=\prod_{j=1}^n \frac{x_j^{\alpha_j} e^{-x_j}}{\Gamma(\alpha_j+1)} dx$ a non-doubling measure defined on $\zinfn$, see \cite[\S4.21]{Leb} for the orthonormality of the family.

    We define, for each $k\in \mathbb N^n$, $\mathcal{L}_k^\alpha (x)=\prod_{i=1}^n L_{k_i}^{\alpha_i}(x_i^2)$, $x=(x_1,\dots,x_n)\in \zinfn$. The sequence $\{\mathcal{L}_k^\alpha\}_{k\in \mathbb N^n}$ is an orthonormal basis for $L^2(\zinfn, \mu_\alpha)$ where \[d\mu_\alpha (x)=2^n\prod_{j=1}^n \frac{x_j^{2\alpha_j+1} e^{-x_j^2}}{\Gamma(\alpha_j+1)} dx\]
    is the pull-back measure from $d\nu_\alpha$ on $\zinfn$ through the one-to-one and onto change of variables $\Psi:\mathbb{R}_+^n\to \mathbb{R}_+^n$ defined as  $\Psi (x)=x^2:=(x_1^2,\cdots, x_n^2),$ for $x=(x_1,\cdots, x_n)\in \mathbb{R}_+^n.$

    We consider the differential Laguerre operator defined on $\zinfn$ as follows
    \[\amsbb{\Delta}_\alpha=-\frac14 \sum_{j=1}^n \left(\frac{d^2}{dx_j^2}+\left(\frac{2\alpha_j+1}{x_j}-2x_j\right)\frac{d}{dx_j}\right).\]
    It turns out that the polynomials $\mathcal{L}_k^\alpha$ are eigenfunctions of the operator $\amsbb{\Delta}_\alpha$, with $\amsbb{\Delta}_\alpha \mathcal{L}_k^\alpha=\widehat{k} \mathcal{L}_k^\alpha$ for every $k\in \mathbb N^n.$

    For every $f\in L^2(\zinfn, \mu_\alpha)$ and $k\in \mathbb N^n$, we denote
    \[c_k^\alpha (f)=\int_{\zinfn} \mathcal{L}_k^\alpha (x) f(x) d\mu_\alpha(x).\]
    We define the operator $\Delta_\alpha$ by
    \[\Delta_\alpha f= \sum_{k\in \mathbb N^n} \lambda_k c_k^\alpha (f) \mathcal{L}_k^\alpha, \quad f\in D(\Delta_\alpha),\]
    where $\lambda_k=\widehat{k}$ for every $k=(k_1,\dots, k_n)\in \mathbb N^n$, and
    \[D(\Delta_\alpha)=\left\{f\in L^2(\zinfn, \mu_\alpha): \sum_{k\in \mathbb N^n} \left|\lambda_k c_k^\alpha(f)\right|^2<\infty\right\},\]
    is the domain of $\Delta_\alpha$ on $L^2(\zinfn,\mu_\alpha).$
    Note that $\Delta_\alpha f=\amsbb{\Delta}_\alpha f$ for every $f\in C_c^\infty (\zinfn)$, the space of smooth and compactly supported functions on $\zinfn$.

    The operator $\Delta_\alpha$ is symmetric and positive, and $-\Delta_\alpha$ generates a semigroup of operators $\{W_t^\alpha\}_{t>0}$ in $L^2(\zinfn, \mu_\alpha)$ where, for every $t>0$,
    \[W_t^\alpha (f)=\sum_{k\in \mathbb N^n} e^{-\lambda_k t} c_k^\alpha (f) \mathcal{L}_k^\alpha, \quad f\in L^2(\zinfn, \mu_\alpha).\]
    According to the Hille-Hardy formula (\cite[(4.17.6)]{Leb} with $x$ and $y$ replaced by $x^2$ and $y^2$ respectively, and $t$ by $e^{-t}$), we have that
    \begin{align*}
        \sum_{k\in \mathbb N^n} e^{-\lambda_k t} \mathcal{L}_k^\alpha(x) \mathcal{L}_k^\alpha(y)&=\prod_{j=1}^n\frac{\Gamma(\alpha_j+1)}{1-e^{-t}} \left(e^{-t/2} x_jy_j\right)^{-\alpha_j} I_{\alpha_j}\left(\frac{2e^{-t/2}x_jy_j}{1-e^{-t}}\right)\\
        &\qquad \times \exp\left(-\frac{e^{-t}}{1-e^{-t}} (x_j^2+y_j^2)\right),
    \end{align*}
    for $x=(x_1,\dots, x_n), y=(y_1,\dots, y_n) \in \zinfn$ and $t>0$. Here  $I_{\nu}$ is the modified Bessel function of the first kind and order $\nu>-1$.  We can write, for every $f\in L^2(\zinfn, \mu_\alpha)$ and $t>0$,
    \begin{equation}\label{Wtalfa}
        W_t^\alpha (f)(x)=\int_{\zinfn} W_t^\alpha (x,y) f(y) d\mu_\alpha (y), \quad x\in \zinfn,
    \end{equation}
    being
    \[W_t^\alpha (x,y)=\prod_{j=1}^n \frac{\Gamma(\alpha_j+1)}{1-e^{-t}}\left(e^{-t/2} x_jy_j\right)^{-\alpha_j} I_{\alpha_j}\left(\frac{2e^{-t/2}x_jy_j}{1-e^{-t}}\right) e^{-\frac{e^{-t}}{1-e^{-t}} (x_j^2+y_j^2)},\]
    for $x=(x_1,\dots, x_n), y=(y_1,\dots, y_n) \in \zinfn$ and $t>0.$
    The integral in \eqref{Wtalfa} defines, for every $t>0$, a contraction on $L^p(\zinfn, \mu_\alpha)$, for every $1\leq p\leq \infty$. By defining, for each $t>0$, $W_t^\alpha$ by \eqref{Wtalfa}, the family $\{W_t^\alpha\}_{t>0}$ is a symmetric diffusion semigroup in Stein's sense in $(\zinfn, \mu_\alpha)$ (see \cite[p.~65]{StLP}).

    The Poisson semigroup $\{P_t^\alpha\}_{t>0}$ associated with the operators $-\sqrt{\Delta_\alpha}$ is defined by
    \[P_t^\alpha(f)=\sum_{k\in \mathbb N^n} e^{-t\sqrt{\lambda_k}} c_k^\alpha(f) \mathcal{L}_k^\alpha, \quad f\in L^2(\zinfn, \mu_\alpha) ,  t>0.\]
    By using the subordination formula, we have that
    \begin{equation}\label{Ptalfa}
        P_t^\alpha(f)=\frac{t}{2\sqrt{\pi}} \int_0^\infty \frac{e^{-\frac{t^2}{4u}}}{u^{\frac32}} W_u^\alpha (f) du, \quad f\in L^2(\zinfn, \mu_\alpha) ,  t>0.
    \end{equation}
    We can write, for every $t>0$ and $f\in L^2(\zinfn, \mu_\alpha)$,
    \begin{equation}\label{Ptalfax}
        P_t^\alpha(f)(x)=\int_{\zinfn} P_t^\alpha (x,y) f(y) d\mu_\alpha(y), \quad x\in \zinfn,
    \end{equation}
    where
    \[P_t^\alpha(x,y)=\frac{t}{2\sqrt{\pi}} \int_0^\infty \frac{e^{-\frac{t^2}{4u}}}{u^{\frac32}} W_u^\alpha (x,y) du, \quad x,y\in \zinfn , t>0.\]
    For each $t>0$, the integral in \eqref{Ptalfax} defines a contraction on $L^p(\zinfn, \mu_\alpha)$ when $1\leq p\leq \infty$. By defining $P_t^\alpha$ as in \eqref{Ptalfa}, $\{P_t^\alpha\}_{t>0}$ is a Stein symmetric diffusion semigroup in $(\zinfn, \mu_\alpha)$.

    The study of harmonic analysis in the Laguerre setting was initiated by Muckenhoupt (\cite{Mu1, Mu2}). Muckenhoupt's context is transferred to ours by applying the transform mapping $\Psi$ mentioned above (see, for instance, \cite{Sa1}).

    The maximal operators $W_*^\alpha$ and $P_*^\alpha$ are defined by
    \[W_*^\alpha(f)=\sup_{t>0}|W_t^\alpha(f)|, \quad P_*^\alpha(f)=\sup_{t>0}|P_t^\alpha(f)|.\]
    From \cite[p.~73]{StLP}, it follows that both $W_*^\alpha$ and $P_*^\alpha$ are bounded on $L^p(\zinfn, \mu_\alpha)$ for every $1<p\leq \infty$. Muckenhoupt (\cite{Mu1}) proved that $W_*^\alpha$ is bounded from $L^1(\zinf,\mu_\alpha)$ into $L^{1,\infty}(\zinf,\mu_\alpha)$. He considered the one-dimensional case. This result was extended to higher dimensions by Dinger (\cite{Di}). Note that the subordination formula \eqref{Ptalfa} allows us to deduce the $L^p$-boundedness properties for $P_*^\alpha$ from the corresponding ones of $W_*^\alpha$. The holomorphic Laguerre semigroups and the maximal operators associated with them where studied in \cite{Sa3}.

   Taking into account the spectral decomposition of $\Delta_\alpha$ and \cite[\S7.2]{NS} we define the first order Riesz-Laguerre transform associated to $\Delta_\alpha$ as \[R_\alpha^i f=\sum_{k\in \mathbb N^n\setminus \{(0,\dots,0)\}} \frac{1}{\sqrt{\lambda_k}} \partial_{x_i} \mathcal{L}_k^\alpha(x) c_k^\alpha(f), \qquad f\in L^2(\zinfn,\mu_\alpha).\]
    Thus the operator $R_\alpha^i$ turns out to be bounded on $L^2(\zinfn,\mu_\alpha).$

   Moreover, we can also define the higher order Riesz-Laguerre transforms as an extension of the first order ones in the following way
   \begin{equation*} R_\alpha^\beta f= \sum_{k\in \mathbb N^n\setminus \{(0,\dots,0)\}} \frac{1}{\lambda_k^{\widehat{\beta}/2}} D^\beta_x \mathcal{L}_k^\alpha(x) c_k^\alpha(f),
   \end{equation*}
    with $\beta\in \mathbb N^n\setminus \{(0,\dots,0)\}$ and $D^\beta_x=\frac{\partial^{\widehat{\beta}}}{\partial x_1^{\beta_1}\dots\ \partial x_n^{\beta_n}}.$ They are also bounded on $L^2(\zinfn,\mu_\alpha),$ see \cite{NS}.
    Let us remark that $R_\alpha^i f=R_\alpha^{e_i}f$ with $e_i$ the $i$-th unit vector and $f\in L^2(\zinfn, \mu_\alpha).$

For every $b>0$ we define the fractional integral $\Delta_\alpha^{-b}$ as the $-b$ power of $\Delta_\alpha$, given, for every $f\in L^2(\zinfn,\mu_\alpha)$, by
\[
\Delta_\alpha^{-b} f=\sum_{k\in \mathbb{N}^n\setminus\{0\}}\lambda_k^{-b}c_k^\alpha(f)\mathcal{L}_k^\alpha.
\]
Let us notice that for any number $b>0$ in \cite{Sa1} it was proved that the integral kernel for $\Delta_\alpha^{-b}$ is given by \begin{align*}
   K_b (x,y)
   &=\frac{1}{\Gamma(b)}\int_0^\infty t^{b-1}\left(W_t^\alpha(x,y)-1\right)dt\\ &=\frac{1}{\Gamma(b)}\int_0^1 (-\log r)^{b-1}\left(W^\alpha_{-\log r}(x,y)-1\right)\, \frac{dr}{r}.
   \end{align*}

 If $f\in C_c^\infty(\zinfn)$ we have that, for every $\beta\in \mathbb N^n\setminus \{(0,\dots,0)\}$,
    \[
R_\alpha^\beta f=D^\alpha\Delta_\alpha^{-\widehat{\beta}/2}f.
\]

   According to what was done in \cite{Sa1} we can also conclude that the operator $R_\alpha^\beta$, off the diagonal, is given by the smooth kernel $D_x^\beta K_{\frac{\widehat{\beta}}{2}}(x,y),$ i.e. \[R_\alpha^\beta f (x)=\int_{\zinfn} D_x^\beta K_{\frac{\widehat{\beta}}{2}}(x,y) f(y)\, d\mu_\alpha(y),\] for all $x\notin \text{supp}(f)$ when $f\in C_c^\infty(\zinfn)$.

   From \cite[Theorem~1.1]{BFRS}, \cite[Theorem~1.1]{Sa1} and \cite[Theorem~13]{N}, we deduce that $R_\alpha^\beta$ can be extended from $L^2(\zinfn,\mu_\alpha)\cap L^p(\zinfn,\mu_\alpha)$ to $L^p(\zinfn,\mu_\alpha)$ as a bounded operator on $L^p(\zinfn,\mu_\alpha)$ when $1<p<\infty$. It can also be extended from $L^1(\zinfn,\mu_\alpha)$ into $L^{1,\infty}(\zinfn,\mu_\alpha)$ for $\widehat{\beta}\le 2$ and from $L^1(\zinfn, w \mu_\alpha)$ into $L^{1,\infty}(\zinfn, \mu_\alpha)$, with $w(y)=(1+\sqrt{|y|})^{\widehat{\beta}-2}$,  for $\widehat{\beta}>2$ (see \cite{FSS1}). We continue denoting by $R_\alpha^\beta$ to those extensions. Furthermore, there exists a constant $c_\beta$ such that, for every $f\in L^p(\zinfn,\mu_\alpha)$, $1\le p<\infty$,
    \[R_\alpha^\beta (f)(x)= c_\beta f(x)+\lim_{\epsilon\to 0^+}\int_{y\in\zinfn,\,|x-y|>\epsilon} R_\alpha^\beta(x,y) f(y) d\mu_\alpha(y),\qquad  \text{a.e. } x\in \zinfn,\]
    where
    \begin{equation}\label{Rieszkernelderiv}
        \begin{split}
         R_\alpha^\beta(x,y)&=\frac{1}{\Gamma\left(\frac{\widehat{\beta}}{2}\right)} \int_0^\infty t^{\frac{\widehat{\beta}}{2}-1} D_x^\beta W_t^\alpha (x,y)dt, \\ & =\frac{1}{\Gamma\left(\frac{\widehat{\beta}}{2}\right)} \int_0^1 (-\log r)^{\frac{\widehat{\beta}}{2}-1} D_x^\beta W_{-\log r}^\alpha (x,y)\, \frac{dr}{r},
        \end{split}
    \end{equation}
    for $x,y\in \zinfn$, $x\neq y.$

     We consider the Littlewood--Paley functions $g_\alpha^{\beta,k}$ defined for Poisson semigroups $\{P_t^\alpha\}_{t>0}$ for $k\in \mathbb N$ and $\beta\in \mathbb{N}^n$ such that $k+\widehat{\beta}>0$, as follows
    \begin{equation*}
        g^{\beta,k}_\alpha (f)(x) =
        \left( \int_0^\infty\left|t^{k+\widehat{\beta}} \partial^k_t D^\beta_x
        P_t^\alpha (f)(x)\right|^2 \frac{dt}{t}
        \right)^{1/2},\;
        x\in\zinfn.
    \end{equation*}
    For simplicity, when $\beta=\boldsymbol{0}=(0,\dots, 0)$, we shall write $g_\alpha^k=g_\alpha^{\boldsymbol{0},k}$. According to \cite[Corollary~1]{StLP}, $g_\alpha^k$ is bounded on $L^p(\zinfn, \mu_\alpha)$ for every $k\in \mathbb N$, $k\geq 1$ and $1<p<\infty$. In \cite[Theorem~1.2]{BdL} it was recently proved that $g_\alpha^k$ is bounded  from $L^1(\zinf,\mu_\alpha)$ into $L^{1,\infty}(\zinf,\mu_\alpha)$. Nowak in \cite[Theorems 6 and 7]{N} proved $L^p$-boundedness properties for $1<p<\infty$ for Littlewood-Paley functions associated with Laguerre polynomial expansions in the $\nu_\alpha$-context including one spatial derivative.

    We say that a function $m$ is of Laplace transform type when
    \[m(x)=x\int_0^\infty \phi(y)e^{-xy} dy, \quad x\in \zinf,\]
    being $\phi \in L^\infty (\zinf)$. Given $m$ of Laplace transform type, we define the spectral multiplier for $\Delta_\alpha$, $T_m^\alpha$, associated with $m$ by
    \[T_m^\alpha (f)=\sum_{k\in \mathbb N^n} m(\lambda_k) c_k^\alpha(f) \mathcal{L}_\alpha^k, \quad f\in L^2(\zinfn, \mu_\alpha).\]
    Since $m$ is bounded, $T_m^\alpha$ is bounded on $L^2(\zinfn, \mu_\alpha)$. According to \cite[Corollary~3, p.~121]{StLP}, $T_m^\alpha$ can be extended from $L^2(\zinfn,\mu_\alpha)\cap L^p(\zinfn,\mu_\alpha)$ to $L^p(\zinfn,\mu_\alpha)$ as a bounded operator on $L^p(\zinfn,\mu_\alpha)$ when $1<p<\infty$. In \cite{Sa4} it was established that $T_m^\alpha$ can be extended from $L^2(\zinfn,\mu_\alpha)\cap L^1(\zinfn,\mu_\alpha)$ to $L^1(\zinfn,\mu_\alpha)$ as a bounded operator from $L^1(\zinfn,\mu_\alpha)$ into $L^{1,\infty}(\zinfn,\mu_\alpha)$. From a higher dimension version of \cite[Theorem~1.1]{BCCR} we deduce that, for every $f\in L^p(\zinfn,\mu_\alpha)$ with $1\leq p<\infty$,
    \[T_m^\alpha(f)(x)=\lim_{\epsilon\rightarrow 0^+} \left(\Lambda(\epsilon)f(x)+\int_{\substack{|x-y|>\epsilon,\\ y\in \zinfn}} K_\phi^\alpha (x,y) f(y)d\mu_\alpha(y)\right), \quad \text{a.e. }x\in \zinfn,  \]
    where $\Lambda\in L^\infty (\zinf)$ and
    \[K_\phi^\alpha(x,y)=\int_0^\infty \phi(t) \left(-\frac{\partial}{\partial t}\right)W_t^\alpha (x,y) dt, \quad x,y\in \zinfn, x\neq y.\]
    A special case of multiplier of Laplace transform type is the imaginary power $\Delta_\alpha^{i\beta}$ of $\Delta_\alpha$ that appears when $m(x)=x^{i\beta}$, for $x\in \zinf$ and $\beta\in \mathbb R$.

    Our objective is to give conditions on a function $p:\zinfn\rightarrow [1,\infty)$ in order that the operators we have just defined (maximal operators, Riesz transforms, Littlewood--Paley functions  and multipliers of Laplace transform type) are bounded on $L^{p(\cdot)}(\zinfn, \mu_\alpha)$.

    Exhaustive studies about Lebesgue spaces with variable exponent (also called generalized Lebesgue spaces or variable Lebesgue spaces) can be found in the monographs \cite{CUF} and \cite{DHHR}.

    Assume that $p:\zinfn\rightarrow [1,\infty)$ is measurable. We say that a measurable function $f$ on $\zinfn$ belongs to $L^{p(\cdot)}(\zinfn, \mu_\alpha)$ if the modular $\varrho_{p(\cdot),\mu_\alpha}(f/\lambda)$ is finite for some $\lambda>0$, where
    \[\varrho_{p(\cdot),\mu_\alpha}(g)=\int_{\zinfn}|g(x)|^{p(x)}d\mu_\alpha(x).\]
    We define on $L^{p(\cdot)}(\zinfn, \mu_\alpha)$ the Luxemburg norm $\|\cdot\|_{L^{p(\cdot)}(\zinfn, \mu_\alpha)}$ associated with $\varrho_{p(\cdot),\mu_\alpha}$, that is,
    \[\|f\|_{L^{p(\cdot)}(\zinfn, \mu_\alpha)}=\inf\left\{\lambda>0:\varrho_{p(\cdot),\mu_\alpha}\left(\frac f\lambda\right)\leq 1 \right\}.\]
    The space $\left(L^{p(\cdot)}(\zinfn, \mu_\alpha),\|\cdot\|_{L^{p(\cdot)}(\zinfn, \mu_\alpha)} \right)$ is a Banach function space. The variable Lebesgue space $L^{p(\cdot)}(\zinfn):=L^{p(\cdot)}(\zinfn, dx)$ and its norm $\|\cdot\|_{L^{p(\cdot)}(\zinfn)}:=\|\cdot\|_{L^{p(\cdot)}(\zinfn, dx)}$ are defined in the obvious way. 

    Lebesgue spaces with variable exponents appear associated to physics problems, image processing and modeling of electrorheological fluids (see, for instance, \cite{AM}, \cite{CLR} and \cite{Ru}).

    As it is well-known, the Hardy--Littlewood maximal function $\mhl$ plays a central role in the study of $L^p$-boundedness properties of harmonic analysis operators. The following conditions on the exponent $p(\cdot)$ arise related with the boundedness of $\mhl$ on $L^{p(\cdot)}(\mathbb R^n)$ (\cite{CUFN} and \cite{Die}):
    \begin{enumerate}[label=(\alph*)]
        \item Local log-H\"older condition: a measurable function $p:\Omega\subset \mathbb R^n\rightarrow [1,\infty)$ is said to be in $\textup{LH}_0(\Omega)$ if there exists $C>0$ such that
        \[|p(x)-p(y)|\leq \frac{C}{-\log|x-y|}, \quad x,y\in \Omega, 0<|x-y|<\tfrac12.\]
        \item Decay log-H\"older condition: a measurable function ${p:\Omega\subset \mathbb R^n\rightarrow [1,\infty)}$ is said to be in $\textup{LH}_\infty(\Omega)$ when there exists $C>0$ and $p_\infty\geq 1$ such that
        \[|p(x)-p_\infty|\leq \frac{C}{\log(e+|x|)}, \quad x \in \Omega.\]
    \end{enumerate}
    We define $\textup{LH}(\Omega)=\textup{LH}_0(\Omega)\cap \textup{LH}_\infty(\Omega)$, where $\Omega\subset \mathbb R^n.$

    If $p:\Omega\subset \mathbb R^n\rightarrow [1,\infty)$ is measurable, we denote by $p^-=\essinf_{\Omega} p$ and ${p^+=\esssup_{\Omega} p}$ the essential infimum and supremum of $p$ on $\Omega$, respectively.

    If $1< p^-\leq p^+<\infty$ and $p\in \textup{LH}(\mathbb R^n)$, then the Hardy--Littlewood maximal function is bounded on $L^{p(\cdot)}(\mathbb R^n)$ (\cite{CUFN}). However, $p\in  \textup{LH}(\mathbb R^n)$  is not necessary for this boundedness (\cite[Examples~4.1~and~4.43]{CUF}). The same conditions on $p$, $1< p^-\leq p^+<\infty$ and $p\in \textup{LH}(\mathbb R^n)$, assure that the Calder\'on--Zygmund singular integrals are bounded on $L^{p(\cdot)}(\mathbb R^n)$ (\cite[Theorem~5.39]{CUF}).

    In \cite{DS1}, Dalmasso and Scotto studied Riesz transforms in the Gaussian setting on variable Lebesgue spaces. In order to do this, they introduced a new class of exponents which is contained in $\textup{LH}_\infty(\mathbb R^n)$. A measurable function $p:\Omega\subset \mathbb R^n \rightarrow [1,\infty)$ is said to be in $\mathcal P^\infty_e(\Omega)$ when there exists $C>0$ and $p_\infty\geq 1$ such that
    \[|p(x)-p_\infty|\leq \frac{C}{|x|^2}, \quad x \in \Omega\setminus \{(0,\dots,0)\}.\]
    If $p_\infty\geq 1$, $A>0$ and $q\geq 2$ are given, the functions $p(x)=p_\infty+\frac{A}{(e+|x|)^q}$, for $x\in \mathbb R^n$, are in $\mathcal{P}^{\infty}_e(\mathbb R^n)$. Main properties of the functions in $\mathcal P^\infty_e(\mathbb R^n)$ were established in \cite{DS1}. Maximal operators defined by the heat semigroup (\cite{MPU}) and Riesz type singular integrals (\cite{DS2} and \cite{NPU}) associated with the Ornstein-Uhlenbeck differential operator were studied on $L^{p(\cdot)}(\mathbb R^n, \gamma_n)$ with $p\in \textup{LH}_0(\mathbb R^n)\cap \mathcal P^\infty_e(\mathbb R^n)$, where $d\gamma_n$ denotes the Gaussian measure.

    We now state the main results of this article concerning $L^{p(\cdot)}$-boundedness properties of harmonic analysis operators in the Laguerre setting.
    \begin{thm}\label{mainthm} Let $\alpha \in [0,\infty)^n$. Assume that $p\in \textup{LH}_0(\zinfn)\cap \mathcal P^\infty_e(\zinfn)$ with $1< p^-\leq p^+<\infty$. We denote by $T_\alpha$ one of the following operators:
    \begin{enumerate}[label=(\alph*)]
        \item \label{max}The maximal operators $W_*^\alpha$ and $P_*^\alpha$;
        \item The Laguerre-Riesz transformation $R^{\beta}_\alpha$, $\beta\in \mathbb N^n\setminus \{(0,\dots,0)\}$;
        \item The Littlewood--Paley functions $g_\alpha^{\beta, k}$ associated with the Poisson semigroup $\{P_t^\alpha\}_{t>0}$, where $\beta \in \mathbb{N}^n$ and $k\in \mathbb N$,  such that $k + \widehat{\beta}>0$;
        \item \label{multipl} The Laguerre spectral multipliers $T_m^\alpha$, where $m$ is a Laplace transform type function.
    \end{enumerate}
    Then, $T_\alpha$ is bounded on $L^{p(\cdot)}(\zinfn, \mu_\alpha)$.
    \end{thm}

    Hereinafter, we prove Theorem~\ref{mainthm}. In Section~\ref{sec2: method}, we explain the method we develop in order to prove that the operators given in \ref{max}--\ref{multipl} are bounded on $L^{p(\cdot)}(\zinfn, \mu_\alpha)$. In Section~\ref{sec3: aux op}, we introduce a global operator that will be a key ingredient for proving our main theorem.  In the following sections, we establish the $L^{p(\cdot)}$-boundedness for each class of operators. Our method exploits the decomposition of the operators into a local part and a global part, which is usual in the study of harmonic analysis in the Laguerre setting, but we need a careful adaptation to the variable exponent context.

    Throughout this paper, $C$ and $c$ will always denote positive constants that may change in each occurrence.

    \section{The method for proving our results}\label{sec2: method}

    In this section we describe the method we apply to prove the boundedness results.

    The polynomial measure $\mathfrak{m}_\alpha$ on $\zinfn$ defined by $d \mathfrak{m}_\alpha(x)=\prod_{i=1}^n x_i^{2\alpha_i+1} dx_i$ is doubling on $\zinfn$. Thus, the triple $(\zinfn, |\cdot|, \mathfrak{m}_\alpha)$ is a homogeneous space in the sense of Coifman and Weiss (\cite{CW}).

    Let $X$ be a Banach space. Suppose that $K:\zinfn \times \zinfn\setminus D\rightarrow X$ is a strongly measurable function, where $D=\{(x,x): x\in \zinfn\}$, satisfying the following two conditions:
    \begin{enumerate}[label=(\roman*)]
        \item \label{size} \textit{Size condition:} there exists $C>0$ such that
        \[\|K(x,y)\|_{X}\leq \frac{C}{\mathfrak{m}_\alpha(B(x,|x-y|))}, \quad x,y\in \zinfn, x\neq y; \]
        \item \label{reg} \textit{Regularity condition:} there exists $C>0$ such that
        \[\|K(x,y)-K(z,y)\|_{X}\leq \frac{C|x-z|}{|x-y|\ \mathfrak{m}_\alpha(B(x,|x-y|))} \]
        and
        \[\|K(x,y)-K(x,z)\|_{X}\leq \frac{C|y-z|}{|x-y|\ \mathfrak{m}_\alpha(B(x,|x-y|))} \]
        for every $x,y,z\in \zinfn$ with $|x-z|\leq \frac12 |x-y|$.
    \end{enumerate}
    When the function $K$ verifies \ref{size} and \ref{reg}, we say that $K$ is an $X$-valued Calder\'on--Zygmund kernel with respect to the homogeneous space $(\zinfn, |\cdot|, \mathfrak{m}_\alpha)$ in the Banach space $X$.

    For every exponent $q:\zinfn \rightarrow [1,\infty)$, we denote by $L^{q(\cdot)}_{X}(\zinfn, \mathfrak{m}_\alpha)$ the $X$-Bochner Lebesgue space with variable exponent $q$, defined in the natural way.

    Assume $T$ is a bounded operator from $L^2(\zinfn,\mathfrak{m}_\alpha)$ into $L^2_{X}(\zinfn,\mathfrak{m}_\alpha)$. We say that $T$ is an $X$-valued Calder\'on--Zygmund operator associated with the Calder\'on--Zygmund kernel $K$ when, for every $f\in C_c^\infty (\zinfn)$,
    \[Tf(x)=\int_{\zinfn} K(x,y) f(y)d\mathfrak{m}_\alpha(y), \quad \text{ a.e. }x\notin \supp(f).\]
    Here, the integral is understood in the $X$-Bochner sense.

    According to \cite[Theorem~1.1]{GLY} (see also \cite{RRT}), if $T$ is an $X$-valued Calder\'on--Zygmund operator on $(\zinfn, |\cdot|, \mathfrak{m}_\alpha)$, $T$ can be extended, for every $1\leq p<\infty$, from $L^2(\zinfn, \mathfrak{m}_\alpha)\cap L^p(\zinfn, \mathfrak{m}_\alpha)$ to $L^p(\zinfn, \mathfrak{m}_\alpha)$ as a bounded operator from $L^p(\zinfn, \mathfrak{m}_\alpha)$ into $L^p_{X}(\zinfn, \mathfrak{m}_\alpha)$ when $1<p<\infty$, and from $L^1(\zinfn, \mathfrak{m}_\alpha)$ into $L^{1,\infty}_{X}(\zinfn, \mathfrak{m}_\alpha)$ when $p=1$.

    Any non-negative measurable function $w$ on $\zinfn$ is named a weight. For every $1<p<\infty$, we say that a weight $w$ on $\zinfn$ is in the Muckenhoupt class $A_p(\zinfn, \mathfrak{m}_\alpha)$ when
    \[\sup_B \left(\frac{1}{\mathfrak{m}_\alpha(B)} \int_B w(x)d\mathfrak{m}_\alpha(x)\right) \left(\frac{1}{\mathfrak{m}_\alpha(B)}\int_B w(x)^{-\frac{1}{p-1}}d\mathfrak{m}_\alpha(x) \right)^{p-1}<\infty,\]
    where the supremum is taken over all the balls $B$ in $\zinfn$.

    A weight $w$ on $\zinfn$ is said to be in the Muckenhoupt class $A_1(\zinfn, \mathfrak{m}_\alpha)$ when there exists $C>0$ such that, for every ball $B\subset \zinfn$,
    \[\frac{1}{\mathfrak{m}_\alpha(B)} \int_B w(x)d\mathfrak{m}_\alpha(x)\leq C \essinf_{y\in B} w(y).\]
    We also define $A_\infty(\zinfn, \mathfrak{m}_\alpha)=\bigcup_{p\geq 1} A_p(\zinfn, \mathfrak{m}_\alpha)$.

    If $T$ is an $X$-valued Calder\'on--Zygmund operator on $(\zinfn, |\cdot|, \mathfrak{m}_\alpha)$, for every $w\in A_p(\zinfn, \mathfrak{m}_\alpha)$ and $1<p<\infty$, the operator $T$ can be extended from $L^2(\zinfn, \mathfrak{m}_\alpha)\cap L^p(\zinfn,w, \mathfrak{m}_\alpha)$ to $L^p(\zinfn, w,\mathfrak{m}_\alpha)$ as a bounded operator from $L^p(\zinfn,w, \mathfrak{m}_\alpha)$ into $L^p_{X}(\zinfn,w, \mathfrak{m}_\alpha)$ (see, for instance, \cite[Theorem~1.1]{Lo}).

    Rubio de Francia's extrapolation theorem works for spaces of homogeneous type (\cite[Theorem~3.5]{AD}). The arguments in the proof of \cite[Theorem~1.3]{CUFMP} allow us to deduce that if $T$ is an $X$-valued Calder\'on--Zygmund operator on $(\zinfn,|\cdot|, \mathfrak{m}_\alpha)$, $T$ defines a bounded operator from $L^{p(\cdot)}(\zinfn, \mathfrak{m}_\alpha)$ into $L_X^{p(\cdot)}(\zinfn, \mathfrak{m}_\alpha)$, provided that $1<p^-\leq p^+<\infty$ and the $\mathfrak{m}_\alpha$-Hardy--Littlewood maximal function is bounded on $L^{p(\cdot)}(\zinfn, \mathfrak{m}_\alpha)$ (see also \cite[Theorem~4.8]{DR}). We recall that according to \cite[Theorems~1.4~and~1.7]{AHH}, the Hardy--Littlewood maximal operator defined by the measure $\mathfrak{m}_\alpha$ is bounded on $L^{p(\cdot)}(\zinfn, \mathfrak{m}_\alpha)$ provided $1<p^-\leq p^+<\infty$ and $p\in \textup{LH}(\zinfn)$ (see also \cite[Theorem~5.2]{DS2}). We also notice that $T$ is well-defined for $f\in L^{p(\cdot)}(\zinfn, \mathfrak{m}_\alpha)$ thanks to the embedding $L^{p(\cdot)}(\zinfn, \mathfrak{m}_\alpha)\hookrightarrow L^{p^-}(\zinfn, \mathfrak{m}_\alpha)+L^{p^+}(\zinfn, \mathfrak{m}_\alpha)$ (\cite[Theorem~3.3.11]{DHHR}).

    The maximal operators and the Littlewood--Paley function can be studied by using Banach valued operators. Indeed, we can write
    \[P_*^\alpha (f)=\|P_t^\alpha(f)\|_{L^\infty(\zinf)}, \qquad W_*^\alpha (f)=\|W_t^\alpha(f)\|_{L^\infty(\zinf)}\]
    and
    \[g_\alpha^{\beta,k}(f)=\left\|t^{k+\widehat{\beta}}\partial_t^k D_x^\beta P_t^\alpha(f) \right\|_{L^2\left(\zinf,\frac{dt}{t}\right)}.\]

    We define
    \[q_\pm (x,y,s)=\sum_{i=1}^n (x_i^2+y_i^2\pm 2x_i y_is_i),\]
    with $x=(x_1,\dots,x_n), y=(y_1,\dots,y_n)\in \zinfn$ and $s=(s_1,\dots,s_n)\in (-1,1)^n$.
    We split $\zinfn {\times \zinfn\times (-1,1)^n}$ into two parts.  Let $\tau>0$ and let us fix $C_0>0$ whose exact value will be specified later. The local part $L_\tau$ is defined by
    \[L_\tau=\left\{(x,y,s)\in \zzia : \sqrt{q_-(x,y,s)}\leq \frac{C_0\tau}{1+|x|+|y|}\right\}\]
    and the global part $G_\tau$ is given by\label{Gtau}
    \[G_\tau=\zinfn {\times \zinfn\times (-1,1)^n} \setminus L_\tau.\]
    By taking into account the integral representation for the modified Bessel function $I_\nu$, $\nu>-\frac12$ (\cite[(5.10.22)]{Leb}), for every $t>0$, the integral kernel of $W_t^\alpha$ can be written as
    \[W_t^\alpha (x,y)=\frac{1}{(1-e^{-t})^{n+\widehat{\alpha}}}\int_{(-1,1)^n} \exp\left(-\frac{q_-\left(e^{-t/2}x,y,s\right)}{1-e^{-t}}+|y|^2\right) \Pi_\alpha(s) ds,\]
    for $x,y\in\zinfn$, where $\widehat{\alpha}=\sum_{i=1}^n\alpha_i$ and $\Pi_\alpha(s)=\prod_{i=1}^n \frac{\Gamma(\alpha_i+1)}{\Gamma(\alpha_i+1/2)\sqrt{\pi}}(1-{s_i^2})^{\alpha_i-1/2}$ for $s=(s_1,\dots,s_n)\in (-1,1)^n$.

    As in \cite{Sa1}, we consider a smooth function $\varphi$ on $\zzia$ such that $0\leq \varphi\leq 1$,
    \[\varphi(x,y,s)=\begin{cases} 1, & (x,y,s)\in L_1,\\ 0, &(x,y,s)\notin L_2,\end{cases}\]
    and
    \[|\nabla_x \varphi(x,y,s)|+|\nabla_y \varphi(x,y,s)|\leq \frac{C}{q_-(x,y,s)^{1/2}}, \quad x,y\in \zinfn, s\in (-1,1)^n.\]

    We also define, for $x,y\in \zinfn$ and $t>0$,
    \[W_{t,\loc}^\alpha (x,y)=\int_{(-1,1)^n} \frac{\exp\left(-\frac{q_-\left(e^{-t/2}x,y,s\right)}{1-e^{-t}}+|y|^2\right)}{(1-e^{-t})^{n+\widehat{\alpha}}} \Pi_\alpha(s)\varphi(x,y,s) ds\]
    and
    \[W_{t,\glob}^\alpha (x,y)=W_t^\alpha(x,y)-W_{t,\loc}^\alpha (x,y).\]

    Suppose that $T_\alpha$ is one of the operators considered in Theorem~\ref{mainthm}. This operator is defined by using the heat integral kernel $W_t^\alpha(x,y)$. We decompose the operator $T_\alpha$ as
    \[|T_\alpha|\le |T_{\alpha,\loc}|+|T_{\alpha,\glob}|,\]
    where $T_{\alpha,\loc}$ is defined as $T_\alpha$ but replacing $W_t^\alpha(x,y)$ by $W_{t,\loc}^\alpha(x,y)$, and in $T_{\alpha, \glob}$ the kernel $W_t^\alpha(x,y)$ is replaced by $W_{t,\glob}^\alpha(x,y)$.

    We shall prove that both $T_{\alpha,\loc}$ and $T_{\alpha,\glob}$ are bounded on $L^{p(\cdot)}(\zinfn,\mu_\alpha)$ provided that $p$ satisfies the hypotheses imposed on Theorem~\ref{mainthm}.

    In order to prove the $L^{p(\cdot)}$-boundedness of $T_{\alpha,\glob}$, we introduce, for every $\varepsilon\in [0,1)$, a positive measurable function $H_{\alpha,\varepsilon}$ defined on $\zz$ verifying that  the operator $\mathcal{H}_{\alpha,\varepsilon}$ given by
        \[\mathcal H_{\alpha,\varepsilon} (f)(x)=\int_{\zinfn} H_{\alpha,\varepsilon}(x,y)f(y)d\mu_\alpha(y), \quad x\in \zinfn\]
        is bounded on $L^{p(\cdot)}(\zinfn,\mu_\alpha)$. Then, we prove that there exists $\varepsilon\in [0,1)$ for which
    \[\left|T_{\alpha,\glob}f(x)\right|\leq \mathcal H_{\alpha,\varepsilon}(|f|)(x), \quad x\in \zinfn.\]

    Secondly, we prove that $T_{\alpha,\loc}$ is bounded on $L^{p(\cdot)}(\zinfn,\mu_\alpha)$. We consider the following Banach spaces
    \[X(W_*^\alpha)=X(P_*^\alpha)=L^\infty(\zinf),\]
    for every $k\in \mathbb{N}$ and $\beta\in \mathbb{N}^n$ such that $k+\widehat{\beta}>0$,
     \[X\left(g_\alpha^{\beta,k}\right)=L^2\left(\zinf,\frac{dt}{t}\right),\]
    and for every $\beta\in \mathbb{N}\setminus\{0\}$ and every multiplier $m$ of Laplace transform type,
    \[X\left(R_\alpha^\beta\right)=X(T_m)=\mathbb C.\]

    We can write
    \[\left|T_{\alpha,\loc}(f)\right|=\left\|\mathbb T_\alpha (f)\right\|_{X(T_\alpha)}\]
    where, for $x\in \zinfn$,
    \[\mathbb T_\alpha (f)(x)=\int_{\zinfn}\int_{(-1,1)^n} \mathcal{M}_\alpha(x,y,s) \varphi(x,y,s) \Pi_\alpha(s)ds \ f(y)d\mathfrak{m}_\alpha(y).\]
    Here, the function $\mathcal{M}_\alpha : \zzia \rightarrow X(T_\alpha)$ is strongly measurable and the integral is understood in the $X(T_\alpha)$-Bochner sense. We write
    \[\mathbb{M}_\alpha(x,y)=\int_{(-1,1)^n} \mathcal{M}_\alpha(x,y,s) \varphi(x,y,s) \Pi_\alpha(s)ds, \quad x,y\in \zinfn.\]
    Thus, $\mathbb{M}_\alpha : \zz\setminus D\rightarrow X(T_\alpha)$ is strongly measurable. 

    The operator $\mathbb T_\alpha$ is bounded from $L^2(\zinfn,\mathfrak{m}_\alpha)$ into $L^2_{X(T_\alpha)}(\zinfn,\mathfrak{m}_\alpha)$. We prove that $\mathbb T_\alpha$ is an $X(T_\alpha)$-valued Calder\'on--Zygmund operator associated with $\mathbb{M}_\alpha$. Then, according to the above-mentioned arguments, $\mathbb T_\alpha$ defines a bounded operator from $L^{p(\cdot)}(\zinfn,\mathfrak{m}_\alpha)$ into $L^{p(\cdot)}_{X(T_\alpha)}(\zinfn,\mathfrak{m}_\alpha)$. We are going to see that $\mathbb T_\alpha$ is also bounded from $L^{p(\cdot)}(\zinfn,\mu_\alpha)$ into $L^{p(\cdot)}_{X(T_\alpha)}(\zinfn,\mu_\alpha)$. Note that the measure $\mu_\alpha$ is not doubling on $(\zinfn,|\cdot|)$.

    As stated in \cite[Lemma~4]{Sa2}, there exists a sequence $\{x(\ell)\}_{\ell \in \mathbb N}\subset \zinfn$ such that, if we set
    \[B_\ell=\left\{x\in \zinfn: |x-x(\ell)|\leq \frac{1}{2(1+|x(\ell)|)}\right\},\quad \ell \in \mathbb N,\]
    the following properties hold
    \begin{enumerate}[label=(\roman*)]
        \item\label{union} $\zinfn=\bigcup\limits_{\ell \in \mathbb N} B_\ell$;
        \item\label{overlap} for every $\delta>1$, the family $\{\delta B_\ell\}_{\ell\in\mathbb N}$ has bounded overlap;
        \item\label{malfa-mualfa} there exists $C>1$ such that, for every $\ell\in\mathbb N$ and every measurable subset $E$ of $B_\ell$,
        \[\frac 1C e^{-|x(\ell)|^2}\mathfrak{m}_\alpha(E)\leq \mu_\alpha(E)\leq Ce^{-|x(\ell)|^2} {\mathfrak{m}_\alpha(E)}.\]
    \end{enumerate}
    Furthermore, for every $\eta>0$, there exists $\delta>1$ such that, if $\ell\in\mathbb N, \,x\in B_\ell$ and $y\notin \delta B_\ell$, then $(x,y,s)\notin L_\eta$ for each $s\in (-1,1)^n$ (see \cite[Remark~5]{Sa2}).

    We have that
    \[\|\mathbb T_\alpha f\|_{L^{p(\cdot)}_{X(T_\alpha)}(\zinfn,\mu_\alpha)}=\left\|\|\mathbb T_\alpha f\|_{X(T_\alpha)}\right\|_{L^{p(\cdot)}(\zinfn,\mu_\alpha)},\]
    so, according to \cite[Corollary~3.2.14]{DHHR},
    \begin{equation*}
    \|\mathbb T_\alpha f\|_{L^{p(\cdot)}_{X(T_\alpha)}(\zinfn,\mu_\alpha)}\leq 2\sup_{\|F\|_{L^{p'(\cdot)}(\zinfn,\mu_\alpha)}\leq 1}\int_{\zinfn} \|\mathbb T_\alpha f(x)\|_{X(T_\alpha)}|F(x)|d\mu_\alpha(x).
    \end{equation*}
    Here, $p'$ denotes the H\"older conjugate exponent of $p$, i.e., $\frac{1}{p(x)}+\frac{1}{p'(x)}=1$ for every $x\in \zinfn$.

    Fix $F\in L^{p'(\cdot)}(\zinfn,\mu_\alpha)$ with $\|F\|_{L^{p'(\cdot)}(\zinfn,\mu_\alpha)}\leq 1$. By virtue of the properties \ref{union}, \ref{overlap} and \ref{malfa-mualfa}, for certain $\delta>1$ we get
    \begin{align*}
        \int_{\zinfn} &\|\mathbb T_\alpha f(x)\|_{X(T_\alpha)}|F(x)|d\mu_\alpha(x)\\
        &\leq \sum_{\ell\in\mathbb N}\int_{B_\ell} \|\mathbb T_\alpha f(x)\|_{X(T_\alpha)}|F(x)|d\mu_\alpha(x)\\
        &=\sum_{\ell\in\mathbb N}\int_{B_\ell} \|\mathbb T_\alpha\left(f\chi_{\delta B_\ell}\right)(x)\|_{X(T_\alpha)}|F(x)|d\mu_\alpha(x)\\
        &\leq C\sum_{\ell\in\mathbb N} e^{-|x(\ell)|^2}\int_{B_\ell} \|\mathbb T_\alpha\left(f\chi_{\delta B_\ell}\right)(x)\|_{X(T_\alpha)}|F(x)|d\mathfrak{m}_\alpha(x)\\
        &\leq C \sum_{\ell\in\mathbb N} e^{-|x(\ell)|^2} \left\|\mathbb T_\alpha(f\chi_{\delta B_\ell})\right\|_{L^{p(\cdot)}_{X(T_\alpha)}(\zinfn,\mathfrak{m}_\alpha)}\|F\chi_{B_\ell}\|_{L^{p'(\cdot)}(\zinfn,\mathfrak{m}_\alpha)}\\
        &\leq C\sum_{\ell\in\mathbb N} e^{-|x(\ell)|^2} \|f\chi_{\delta B_\ell}\|_{L^{p(\cdot)}(\zinfn,\mathfrak{m}_\alpha)}\|F\chi_{B_\ell}\|_{L^{p'(\cdot)}(\zinfn,\mathfrak{m}_\alpha)}.
    \end{align*}
    We have used H\"older's inequality with variable exponents (see, for instance, \cite[Lemma~3.2.20]{DHHR}).

    Since $p\in \mathcal P^\infty_e(\zinfn)$ and $1<p^-\leq p^+<\infty$, we also have 
    $p'\in \mathcal P^\infty_e(\zinfn)$ with $1<(p')^-\leq (p')^+<\infty$. From \cite[Lemma~2.5]{DS1}, by proceeding as in \cite[(3.12)]{DS1} and the following lines, we get
    \[e^{-|x(\ell)|^2/p_\infty} \|f\chi_{\delta B_\ell}\|_{L^{p(\cdot)}(\zinfn,\mathfrak{m}_\alpha)}\leq \|f\chi_{\delta B_\ell}\|_{L^{p(\cdot)}(\zinfn,\mu_\alpha)}\]
    and
    \[e^{-|x(\ell)|^2/p'_\infty} \|F\chi_{B_\ell}\|_{L^{p'(\cdot)}(\zinfn,\mathfrak{m}_\alpha)}\leq \|F\chi_{B_\ell}\|_{L^{p'(\cdot)}(\zinfn,\mu_\alpha)},\]
    where $p'_\infty$ is the conjugate exponent of $p_\infty$.

    By means of \cite[Corollary~2.8]{DS1}, we obtain
    \begin{align*}
         \int_{\zinfn} \|\mathbb T_\alpha f(x)\|_{X(T_\alpha)}&|F(x)|d\mu_\alpha(x)\leq C\sum_{\ell\in \mathbb N} \|f\chi_{\delta B_\ell}\|_{L^{p(\cdot)}(\zinfn,\mu_\alpha)} \|F\chi_{B_\ell}\|_{L^{p'(\cdot)}(\zinfn,\mu_\alpha)}\\
         &\leq C \sum_{\ell\in \mathbb N} \left\|f\chi_{\delta B_\ell} e^{-|\cdot|^2/p(\cdot)}\prod_{i=1}^n \frac{x_i^{(2\alpha_i+1)/p(\cdot)}}{\Gamma(\alpha_i+1/2)^{1/p(\cdot)}} \right\|_{L^{p(\cdot)}(\zinfn)}\\
         &\qquad\qquad \times \left\|F\chi_{B_\ell}e^{-|\cdot|^2/p'(\cdot)}\prod_{i=1}^n \frac{x_i^{(2\alpha_i+1)/p'(\cdot)}}{\Gamma(\alpha_i+1/2)^{1/p'(\cdot)}}\right\|_{L^{p'(\cdot)}(\zinfn)}\\
         &\leq C\|f\|_{L^{p(\cdot)}(\zinfn,\mu_\alpha)}\|F\|_{L^{p'(\cdot)}(\zinfn,\mu_\alpha)}.
    \end{align*}
    Hence, we conclude that
    \[\|\mathbb T_\alpha f\|_{L^{p(\cdot)}_{X(T_\alpha)}(\zinfn,\mu_\alpha)}\leq C\|f\|_{L^{p(\cdot)}(\zinfn,\mu_\alpha)}.\]

    We have thus proved that the operator $T_{\alpha,\loc}$ is bounded on $L^{p(\cdot)}(\zinfn,\mu_\alpha)$ provided that the exponent function $p$ satisfies the conditions of Theorem~\ref{mainthm}.

    \section{An auxiliary result}\label{sec3: aux op}

    In this section we establish a result that will be useful to prove $L^{p(\cdot)}$-boundedness for the global parts of the operators considered in Theorem~\ref{mainthm}.

    Given $\alpha\in [0,\infty)^n$ and $\varepsilon\in [0,1)$, we define the global operator
    \[\mathcal{H}_{\alpha,\varepsilon}(f)(x)=\int_{\zinfn} H_{\alpha,\varepsilon}(x,y) f(y)d \mathfrak{m}_\alpha(y),\quad x\in \zinfn,\]
    where
    \[H_{\alpha,\epsilon}(x,y)=\int_{(-1,1)^n} H_{\alpha,\varepsilon}(x,y,s) (1-\varphi(x,y,s))\Pi_\alpha(s)ds\]
    and
    \begin{equation}\label{H}
    H_{\alpha,\varepsilon}(x,y,s)=\begin{cases}e^{-(1-\varepsilon)|y|^2}, &\hspace{-0.4em}\sum\limits_{i=1}^n x_i y_i s_i\leq 0,\\
    q_+(x,y,s)^{n+\widehat{\alpha}}e^{-\frac{(1-\varepsilon)}{2}\left(|y|^2-|x|^2+\sqrt{q_+(x,y,s)q_-(x,y,s)}\right)}, &\hspace{-0.4em} \sum\limits_{i=1}^n x_i y_i s_i> 0.
    \end{cases}
    \end{equation}

    \begin{prop}\label{propoaux}
    Let $\alpha\in [0,\infty)^n$. Suppose that $p\in \textup{LH}_0(\zinfn)\cap \mathcal{P}_e^\infty(\zinfn)$ with $1<p^-\leq p^+<\infty$ and let $0<\varepsilon<\frac{1}{(p^-)'}\wedge \frac{1}{n+\widehat{\alpha}}$. Then, the operator $\mathcal{H}_{\alpha,\varepsilon}$ is bounded on $L^{p(\cdot)}(\zinfn,\mu_\alpha)$.
    \end{prop}

    \begin{proof} 
    We decompose $\mathcal{H}_{\alpha,\varepsilon}(f)=\mathcal{H}_{\alpha,\varepsilon}^{(1)}(f)+\mathcal{H}_{\alpha,\varepsilon}^{(2)}(f)$, where
    \[\mathcal{H}_{\alpha,\varepsilon}^{(1)}(f)(x)=\int_{E_x} H_{\alpha,\varepsilon}(x,y,s) (1-\varphi(x,y,s))\Pi_\alpha(s)ds f(y)d \mathfrak{m}_\alpha(y),\]
    and
    \[\mathcal{H}_{\alpha,\varepsilon}^{(2)}(f)(x)=\int_{F_x} H_{\alpha,\varepsilon}(x,y,s) (1-\varphi(x,y,s))\Pi_\alpha(s)ds f(y)d \mathfrak{m}_\alpha(y),\]
    being
    \begin{align*}
        E_x&=\left\{(y,s)\in \zia :  \sum_{i=1}^n x_i y_i s_i\leq 0\right\},\\
        F_x&=\left\{(y,s)\in \zia :  \sum_{i=1}^n x_i y_i s_i>0\right\}.
    \end{align*}

    Let $f\in L^{p(\cdot)}(\zinfn, \mu_\alpha)$ be given such that $\|f\|_{L^{p(\cdot)}(\zinfn,\mu_\alpha)}\leq 1$. For $x\in \zinfn$, we have that
    \begin{align*}
        \left|\mathcal{H}_{\alpha,\varepsilon}^{(1)}(f)(x)\right|&\leq \int_{\zinfn} e^{-(1-\varepsilon)|y|^2} |f(y)|\int_{(-1,1)^n} |1-\varphi(x,y,s)|\Pi_\alpha(s) ds d\mathfrak{m}_\alpha(y)\\
        &\leq C\int_{\zinfn} e^{-(1-\varepsilon)|y|^2} |f(y)| d\mathfrak{m}_\alpha(y).
    \end{align*}
    Since $\varepsilon<1/(p^-)'$, we can write $1-\varepsilon=\tilde{\varepsilon}+1/p^-$ with $\tilde{\varepsilon}>0$. Thus, by H\"older's inequality with $p^->1$ we have
    \begin{align*}
        &\left|\mathcal{H}_{\alpha,\varepsilon}^{(1)}(f)(x)\right|\\
        &\quad\leq C\int_{\zinfn} e^{-\left(\tilde{\varepsilon}+\frac{1}{p^-}\right)|y|^2} |f(y)| d\mathfrak{m}_\alpha(y)\\
        &\quad\leq C\left(\int_{\zinfn} e^{-|y|^2} |f(y)|^{p^-} d\mathfrak{m}_\alpha(y)\right)^{1/p^-}\left(\int_{\zinfn} e^{-\tilde{\varepsilon}(p^-)'|y|^2} d\mathfrak{m}_\alpha(y)\right)^{1/(p^-)'}\\
        &\quad\leq C\left(\int_{\zinfn\cap \{|f|>1\}} |f(y)|^{p(y)} d\mu_\alpha(y)+\int_{\zinfn\cap \{|f|\leq 1\}} d\mu_\alpha(y)\right)^{1/p^-}\leq C,
    \end{align*}
    since $\int_{\zinfn} |f(y)|^{p(y)}d\mu_\alpha(y)\leq 1$ and $\mu_\alpha$ is a probability measure on $\zinfn$. 
    
    Therefore, by the homogeneity of the norm,
    \begin{equation*}
        \left\|\mathcal{H}_{\alpha,\varepsilon}^{(1)} (f)\right\|_{L^{p(\cdot)}(\zinfn, \mu_\alpha)}
        \leq C
        \|f\|_{L^{p(\cdot)}(\zinfn, \mu_\alpha)}.
    \end{equation*}
    for any $f\in L^{p(\cdot)}(\zinfn,\mu_\alpha)$.

    We now study $\mathcal{H}_{\alpha,\varepsilon}^{(2)}$. We have that
    \begin{equation*}
    \begin{split}
        &\int_{\zinfn} |\mathcal{H}_{\alpha,\varepsilon}^{(2)}(f)(x)|^{p(x)}d\mu_\alpha(x) \leq
        C \int_{\zinfn} \left(
        \int_{F_x} |f(y)| e^{\frac{-|y|^2}{p(y)}}e^{\frac{|y|^2}{p(y)} - \frac{|x|^2}{p(x)}} \right.
        \\ & \;\times \left.
        q_+(x,y,s)^{n+\widehat{\alpha}}
        e^{-\frac{(1-\varepsilon)}{2}(|y|^2-|x|^2+\sqrt{q_+(x,y,s)q_-(x,y,s)})}\Pi_\alpha(s)ds d\mathfrak{m}_\alpha(y)
        \vphantom{\int}\right)^{p(x)}d\mathfrak{m}_\alpha(x).
    \end{split}
    \end{equation*}
    Note that we can write
    \begin{align*}
            q_+(x,y,s)& q_-(x,y,s)\\
            &
            = \left( |x|^2 + |y|^2 + 2\sum_{i=1}^n x_i y_i s_i
            \right)\left( |x|^2 + |y|^2 - 2\sum_{i=1}^n x_i y_i s_i
            \right)
            \\ & =
            \left( |x|^2 + |y|^2\right)^2
            - 4\left(\sum_{i=1}^n x_i y_i s_i\right)^2
            \\ & =
            |x|^4 + |y|^4 + 2|x|^2 |y|^2
            -4\left(\sum_{i=1}^n x_i y_i s_i\right)^2
            \\ & =
            \left( |x|^2 - |y|^2 \right)^2
            + 4\left( |x|^2 |y|^2 -
            \left(\sum_{i=1}^n x_i y_i s_i\right)^2\right)
            \\ & \geq
            \left( |x|^2 - |y|^2 \right)^2
            + 4\left( |x|^2 |y|^2 - \left|
            \langle (x_1,\dots,x_n),(s_1y_1,\dots,s_n y_n)\rangle
            \right|^2\right)
            \\ & \geq
            \left( |x|^2 - |y|^2 \right)^2
            + 4\left( |x|^2 |y|^2 - |x|^2 |(s_1 y_1, \dots, s_n y_n)|^{2}\right)
            \\ & \geq
            \left( |x|^2 - |y|^2 \right)^2,
        \end{align*}
    for each $x=(x_1,\ldots,x_n),\,y=(y_1,\ldots,y_n)\in \zinfn$ and $s=(s_1,\ldots,s_n)\in (-1,1)^n$.

    On the other hand, according to \cite[Lemma~2.5]{DS1}, since $p\in \mathcal{P}_e^\infty(\zinfn)$ then
    \begin{equation*}
        e^{\frac{|y|^2}{p(y)} - \frac{|x|^2}{p(x)}} \sim
        e^{\frac{|y|^2 - |x|^2}{p_\infty}},
        \quad x,y\in \zinfn.
    \end{equation*}
    Here $p_\infty>1$. Whence, it follows that
    \begin{align*}
            & q_+(x,y,s)^{n+\widehat{\alpha}}
            \exp\left(-\tfrac{(1-\varepsilon)}{2}(|y|^2-|x|^2+\sqrt{q_+(x,y,s)q_-(x,y,s)}\right)
            e^{\frac{|y|^2}{p(y)} - \frac{|x|^2}{p(x)}}
            \\ & \; \leq C
            q_+(x,y,s)^{n+\widehat{\alpha}}
            \exp\left(\left(\tfrac{1}{p_\infty} - \tfrac{1-\varepsilon}{2}\right)
            \left(|y|^2-|x|^2\right)
            -\tfrac{(1-\varepsilon)}{2}\sqrt{q_+(x,y,s)q_-(x,y,s)}\right)
            \\ & \; \leq
            C \left(q_+(x,y,s)
            \right)^{n+\widehat{\alpha}}
            \exp\left(
            -a_\varepsilon \sqrt{q_+(x,y,s)q_-(x,y,s)}
            \right),
        \end{align*}
      for every $x=(x_1,\ldots,x_n), \,y=(y_1,\ldots,y_n)\in \zinfn$ such that $(x,y,s)\in G_1$ and $\sum_{i=1}^n x_i y_i s_i\geq 0$. We recall that \[G_1=\left\{(x,y,s)\in \zzia : \sqrt{q_-(x,y,s)}\geq \frac{C_0}{1+|x|+|y|}\right\}.\] Above we have set $a_\varepsilon= \frac{1-\varepsilon}{2}- \left|\frac{1}{p_\infty} - \frac{1-\varepsilon}{2}\right|$. Note that $a_{\varepsilon}>0$ because $\varepsilon<1/(p^-)'$ and $(p^-)'=(p')^+\geq p'_\infty$. 
      
      We get
  \begin{align*}
            \int_{\zinfn} &  \left| \mathcal{H}_{\alpha,\varepsilon}^{(2)}(f)(x)\right|^{p(x)} d\mu_\alpha(x)
            \\ & \leq C \int_{\zinfn} \left(
            \int_{F_x}|f(y)| e^{\frac{-|y|^2}{p(y)}} |1 - \varphi(x,y,s)|
            q_+(x,y,s)^{n+\widehat{\alpha}}
            \right. \\ & \left. \quad \times
            \exp\left(-a_\varepsilon \sqrt{q_+(x,y,s)q_-(x,y,s)}
            \right) \Pi_\alpha(s)ds d\mathfrak{m}_\alpha(y)
            \right)^{p(x)} d\mathfrak{m}_\alpha(x).
        \end{align*}

    In order to complete the study of $\mathcal{H}_{\alpha,\varepsilon}^{(2)}$ we use Stein complex interpolation.
    We consider firstly $n = 1$. For every $z\in \mathbb{C}$ with $\Real(z) > -\frac12$, we define the operator $\mathbb{H}_{z,\varepsilon}^{(2)}$ by
    \begin{align*}
            \mathbb{H}_{z,\varepsilon}^{(2)}(h)(x)
            & =
            \int_0^{\infty} K_{z,\varepsilon}^{(2)}(x,y) h(y) y^{2z+1} dy \ x^{\frac{2z+1}{p(x)}}
            \\ & = \widehat{\mathcal{H}}_{z,\varepsilon}^{(2)}(h)(x)\  x^{\frac{2z+1}{p(x)}}, \quad x\in \zinf,
        \end{align*}
    where
    \begin{align*}
            K_{z,\varepsilon}^{(2)}(x,y)
            & =
            \int_{-1}^1 \chi_{F_x}(y,s) (1-\varphi(x,y,s))
            (q_+(x,y,s))^{z+1}
            \\ & \quad \times \exp\left(
            -a_\varepsilon\sqrt{q_+(x,y,s)q_-(x,y,s)}\right) (1-s^2)^{z-\frac12}ds, \quad x,y\in \zinf,
        \end{align*}
    and $a_\varepsilon$ is as above.

    For every $z\in \mathbb{C}$ with $\Real(z)>-\frac12$ and every simple function $h$ defined on $(\zinf, dx)$, $\mathbb{H}_{z,\varepsilon}^{(2)}(h)$ is a measurable function on $(\zinf,dx)$.

    Assume that $r$, $y$, $c_1$, $c_2>0$, $b_1$, $b_2$, $m_1$ and $m_2$ are positive bounded measurable functions on $\zinf$, and $A_1$ and $A_2$ are two measurable subsets of $\zinf$ with finite Lebesgue measure. We define
    \begin{equation*}
            F_{y,r}(z) = \int_{B(y,r)} {\mathbb{H}_{z,\varepsilon}^{(2)}}
            \left(c_1^{m_1(\cdot)z+b_1(\cdot)} \chi_{A_1}(\cdot)
            \right)(x)
            c_2^{m_2(x)z+b_2(x)} \chi_{A_2}(x) dx,
    \end{equation*}
    for $z\in \mathbb{C}$, $\Real(z)>-\frac12$. The function $F_{y,r}$ is analytic on $\Omega=\left\{z\in \mathbb C: \Real(z)>-\frac12\right\}$. Furthermore, for every $-\frac12<c<d<\infty$,
    \begin{equation*}
        \sup_{c\leq \Real(z) \leq d} |F_{y,r}(z)| <\infty.
    \end{equation*}
    Thus, the family $\left\{\mathbb{H}_{z,\varepsilon}^{(2)}\right\}_{z\in\Omega}$ is an analytic family of admissible growth in every strip $\{z\in\mathbb{C}:c<\Real(z)<d\}$, with $-\frac12<c<d<\infty$ (see \cite[\S3]{MRA}).

    Let $k\in\mathbb{N}$, $k>1$. We take $\alpha = \frac{k}{2} - 1$. For every $\overline{x}\in\mathbb{R}^k$ we write $x = |\overline{x}|$.
    If $\overline{x}$, $\overline{y}\in \mathbb{R}^k$ and $\theta$ is the angle between $\overline{x}$ and $\overline{y}$, we have that
    \begin{equation*}
    |\overline{x} \pm \overline{y}|^2
    = q_{\pm}(x,y,\cos(\theta)),
    \end{equation*}
    and also that $(x,y,\cos(\theta))\in L_1$ if and only if $|\overline{x}-\overline{y}|<C_0/(1+x+y)$.
    By integrating in spherical coordinates on $\mathbb{R}^k$ and by performing the change of variable $s = \cos(\theta)$ we obtain
    \begin{align*}
            \left| \mathbb{H}_{\alpha,\varepsilon}^{(2)}(h)(x) \right|
            & \leq
            C x^{\frac{k-1}{p(x)}}\,\int_{|\overline{x} - \overline{y}|>\frac{C_0}{1+x+y}}
            |\overline{x} + \overline{y}|^k
            e^{-a_\varepsilon|\overline{x}-\overline{y}||\overline{x}+\overline{y}|} |{h}(y)|d\overline{y}  ,
        \end{align*}
    for $x = |\overline{x}|\in\zinf$. We consider the operators
    \begin{equation*}
        \begin{split}
            T_1(h)(\overline{x})
            &  =
            \int_{\substack{|\overline{x} - \overline{y}|>\frac{C_0}{1+x+y}\\ 2|\overline{x} - \overline{y}| \geq |\overline{x} + \overline{y}| }}
            |\overline{x} + \overline{y}|^k
            e^{-a_\varepsilon|\overline{x}-\overline{y}||\overline{x}+\overline{y}|} h(\overline{y}) d\overline{y},
        \end{split}
    \end{equation*}
    and
    \begin{equation*}
        \begin{split}
            T_2(h)(\overline{x})
            & =
            \int_{\substack{|\overline{x} - \overline{y}|>\frac{C_0}{1+x+y}\\
            2|\overline{x} - \overline{y}| <|\overline{x} + \overline{y}|
            }}
            |\overline{x} + \overline{y}|^k
            e^{-a_\varepsilon|\overline{x}-\overline{y}||\overline{x}+\overline{y}|} h(\overline{y})d\overline{y},
        \end{split}
    \end{equation*}
    for $\overline{x}\in \mathbb{R}^k$. We are going to see that $T_1$ and $T_2$ are bounded on $L^{\overline{p}(\cdot)}(\mathbb{R}^k,dx)$, where $\overline{p}(\overline{x}) = p(|\overline{x}|)$, $\overline{x}\in \mathbb{R}^k$.

    Note firstly that
    \begin{equation*}
        \begin{split}
            |T_1(h)(\overline{x})|
            & \leq C \left(
            \int_{B({-\overline{x}},1)}|h(\overline{y})| d\overline{y} +
            \sum_{\ell = 1}^{\infty} \int_{\ell\leq |\overline{x}{+}\overline{y}|<\ell +1}
            e^{-c|\overline{x}{+}\overline{y}|^ 2}
            |h(\overline{y})| d\overline{y}
            \right)
            \\ & \leq C \sum_{\ell=0}^\infty e^{-c\ell^2}\int_{B(-\overline{x},\ell +1)}|h(\overline{y})| d\overline{y}
            \\ & \leq C
            \mhl(h)({-\overline{x}}), \quad
            \overline{x}\in\mathbb{R}^k.
        \end{split}
    \end{equation*}
    Here, $\mhl$ represents the Hardy--Littlewood maximal function in $\mathbb{R}^ k$.

    On the other hand, according to \cite[(16)~and~(17)]{GMMST}, if $2|\overline{x} - \overline{y}| < |\overline{x} + \overline{y}|$, then $|\overline{y}|\leq 3|\overline{x}|$ and $\frac43|\overline{x}| \leq |\overline{x} + \overline{y}| \leq 4|\overline{x}|$. We obtain
    \begin{equation*}
        \begin{split}
            |T_2(h)(\overline{x}) |
            & \leq
            C \int_{|\overline{x} - \overline{y}|>\frac{C_0}{1+4x}}
            |\overline{x}|^k e^{-c|\overline{x}||\overline{x} - \overline{y}|} |h(\overline{y})| d\overline{y}
            \\ & \leq
           \begin{cases}
            \int_{|\overline{x} - \overline{y}|\leq 4} |h(\overline{y})| d\overline{y} \leq C \mhl(h)({\overline{x}})
            & \text{if}\;\;  |\overline{x}|\leq 1,\\
            \int_{|\overline{x} - \overline{y}|> C_0/(5|\overline{x}|)} |\overline{x}|^k e^{-c|\overline{x}||\overline{x} - \overline{y}|} |h(\overline{y})| d\overline{y}
            & \text{if}\;\; |\overline{x}|> 1.
        \end{cases}
        \end{split}
    \end{equation*}

    Since {$\overline{p}(\overline{x})=\overline{p}(-\overline{x})$, and} under the imposed conditions for $p(\cdot)$, $\mhl$ is bounded on $L^{\overline{p}(\cdot)}(\mathbb{R}^k,dx)$ {(see Lemma~\ref{lem:pbarra} for $n=1$)}, the arguments developed in \cite[pp.~{417~and~418}]{DS1} allow us to conclude that $T_1$ and $T_2$ are bounded on $L^{\overline{p}(\cdot)}(\mathbb{R}^k,dx)$. 

    We have, therefore, that the operator $T{:=T_1+T_2}$ is bounded on $L^{\overline{p}(\cdot)}(\mathbb{R}^k,dx)$.

    Since
    \begin{equation*}
        \left|\mathbb{H}_{\alpha,\varepsilon}^{(2)}(h)(x)\right| \leq C x^{\frac{k-1}{p(x)}}
        T(|\widetilde{h}|)(\overline{x}), \quad x = |\overline{x}|, \quad x\in\mathbb{R}^k,
    \end{equation*}
    we get
    \begin{equation*}
        \begin{split}
            \int_0^\infty \left|
            \mathbb{H}_{\alpha,\varepsilon}^{(2)}(h)(x)\right|^{p(x)}dx
            & =
            \int_0^\infty \left|\widehat{\mathcal{H}}_{\alpha,\varepsilon}^{(2)}(h)(x)
            \right|^{p(x)} x^{k-1}dx \leq C
            \int_{\mathbb{R}^k}
            \left|T({|\widetilde{h}|})\left(|\overline{x}|\right)
            \right|^{\overline{p}(\overline{x})} d\overline{x},
        \end{split}
    \end{equation*}
    where ${\widetilde{h}}(\overline{y}) = {h}(|\overline{y}|)$, $\overline{y}\in \mathbb{R}^k$. Hence
    \begin{equation*}
            \left\|\mathbb{H}_{\alpha,\varepsilon}^{(2)}(h)\right\|_{L^{p(\cdot)}(\zinf,dx)} \leq C
            \left\|T\left(|\widetilde{h}|\right)\right\|_{L^{\overline{p}(\cdot)}(\mathbb{R}^k,dx)}\leq C \left\|{\widetilde{h}}\right\|_{L^{\overline{p}(\cdot)}(\mathbb{R}^k,dx)}.
    \end{equation*}

    Naming $h_k(u)=h(u)u^{\frac{k-1}{p(u)}}$, $u\in \zinf$, we also have 
    \begin{equation*}
        \begin{split}
            \int_{\mathbb{R}^k}  \left|
            {\widetilde{h}}(\overline{x})
            \right|^{\overline{p}(x)}d\overline{x}
            & =
            \int_{\mathbb{R}^k}  \left|
            {h}(|\overline{x}|)\right|^{p(|\overline{x}|)}d\overline{x} = C
            \int_0^\infty
            |{h}(x)|^{p(x)} x^{k-1} dx
            \\ & = C
            \int_0^\infty
            \left|
            {h}(x)x^{\frac{k-1}{p(x)}}
            \right|^{p(x)}  dx{=C\int_0^\infty |h_k(x)|^{p(x)}dx,}
        \end{split}
    \end{equation*}
    {which yields $\|\widetilde{h}\|_{L^{\overline{p}(\cdot)}(\mathbb{R}^k,dx)}\leq C \|h_k\|_{L^{p(\cdot)}(\zinf,dx)}.$}
    We conclude that
    \begin{equation*}
        \left\|\mathbb{H}_{\alpha,\varepsilon}^{(2)}(h)\right\|_{L^{p(\cdot)}(\zinf,dx)}
        \leq C
        \left\|{h_k}\right\|_{L^{p(\cdot)}(\zinf,dx)}.
    \end{equation*}

    We now consider, for every $z\in\mathbb{C}$ with $\Real(z)>-\frac12$,
    \begin{equation*}
        \mathcal{C}_{z,\varepsilon}(f)(x) =
        \mathcal{H}_{z,\varepsilon}^{(2)} ({f_z})(x), \quad x\in\zinf,
    \end{equation*}
    {where $f_z(y)=f(y)y^{-\frac{2z+1}{p(y)}}$, $y\in\zinf$.}

     The family $\{\mathcal{C}_{z,\varepsilon}\}_{\Real(z)>-\frac12}$ is an analytic family of admissible growth in every strip $\{z\in\mathbb{C}:c<\Real(z)<d\}$ with $-\frac12<c<d<\infty$  (\cite[\S3]{MRA}). For every $k\in\mathbb{N}$, $k>1$, we have that
    \begin{equation*}
        \left\| \mathcal{C}_{\frac{k}{2}-1,\varepsilon}(f)
        \right\|_{L^{p(\cdot)}(\zinf,dx)}
        \leq C_0 \left\| f
        \right\|_{L^{p(\cdot)}(\zinf,dx)}
    \end{equation*}
    and{, for each $t\in \mathbb{R}$,}
    \begin{equation*}
        \left\| \mathcal{C}_{\frac{k}{2}-1+it,\varepsilon}(f)
        \right\|_{L^{p(\cdot)}(\zinf,dx)}
        \leq  \left\| \mathcal{C}_{\frac{k}{2}-1,\varepsilon}(|f|)
        \right\|_{L^{p(\cdot)}(\zinf,dx)}\leq C_0 \left\| f
        \right\|_{L^{p(\cdot)}(\zinf,dx)}.
    \end{equation*}

    According to \cite[Theorem~1]{MRA}, 
    for every $\alpha\geq 0$,
    \begin{equation*}
        \left\| \mathcal{C}_{\alpha,\varepsilon}(f)
        \right\|_{L^{p(\cdot)}(\zinf,dx)}
        \leq C_\alpha \left\| f
        \right\|_{L^{p(\cdot)}(\zinf,dx)}.
    \end{equation*}
    It follows that, for every $\alpha\geq 0$,
    \begin{equation*}
        \begin{split}
            \int_0^\infty \left|
            \mathcal{H}_{\alpha,\varepsilon}^{(2)}(f)(x)
            \right|^{p(x)} d\mu_\alpha (x)
            & \leq C
            \int_0^\infty \left|
            \mathbb{H}_{\alpha,\varepsilon}^{(2)}
            \left(f(\cdot)e^{-\frac{|\cdot|^2}{p(\cdot)}}\right)(x)
            \right|^{p(x)} dx
            \\ & = C
            \int_0^\infty \left|
            \mathcal{C}_{\alpha,\varepsilon}
            \left(f(\cdot)e^{-\frac{|\cdot|^2}{p(\cdot)}} (\cdot)^{\frac{2\alpha+1}{p(y)}}\right)(x)
            \right|^{p(x)} dx.
        \end{split}
    \end{equation*}
    Then
    \begin{equation*}
        \begin{split}
            \left\| \mathcal{H}_{\alpha,\varepsilon}^{(2)}(f) \right\|_{L^{p(\cdot)}(\zinf,\mu_\alpha)}
            & \leq C
            \left\| \mathcal{C}_{\alpha,\varepsilon}
            \left(f(\cdot)e^{-\frac{|\cdot|^2}{p(\cdot)}} (\cdot)^{\frac{2\alpha+1}{p(\cdot)}}\right) \right\|_{L^{p(\cdot)}(\zinf,dx)}
            \\ & \leq C
            \left\| f(\cdot)e^{-\frac{|\cdot|^2}{p(\cdot)}} (\cdot)^{\frac{2\alpha+1}{p(\cdot)}}
            \right\|_{L^{p(\cdot)}(\zinf,dx)}
            \\ & \leq C
            \left\| f
            \right\|_{L^{p(\cdot)}(\zinf,\mu_\alpha)}.
        \end{split}
    \end{equation*}
    We conclude that the operator $\mathcal{H}_{\alpha, \varepsilon}$ is bounded on $L^{p(\cdot)}(\zinf,\mu_\alpha)$.

    We now prove that $\mathcal{H}_{\alpha, \varepsilon}$ is bounded on $L^{p(\cdot)}(\zinfn,\mu_\alpha)$ when the dimension $n$ is greater than one.

    Let $n\in\mathbb{N}$, $n>1$. We define
    \begin{equation*}
        \mathbb{H}_{z,\varepsilon}^{(2)}(h)(x) =
        \int_{\zinfn} K_{z,\varepsilon}^{(2)}(x,y)h(y) \prod_{j=1}^n
        y_j^{2z_j+1} dy \prod_{j=1}^n x_j^{\frac{2z_j+1}{p(x)}},
    \end{equation*}
    for $x = (x_1,\dots,x_n)\in\zinfn$ and $z = (z_1,\dots,z_n)\in\mathbb{C}^n$ with $\Real(z_j)>-\frac12$ for each $j=1,\dots,n$, where
    \begin{equation*}
    \begin{split}
        K_{z,\varepsilon}^{(2)}(x,y) & =
        \int_{(-1,1)^n} \chi_{F_x}(y,s)(1-\varphi(x,y,s)) q_+(x,y,s)^{n+\widehat{z}}
        \\ & \qquad \times
        \exp\left(-a_\varepsilon  \sqrt{q_+(x,y,s)q_-(x,y,s)}
        \right)
        \prod_{j=1}^n (1-s_j^2)^{z_j - 1/2}   ds,
    \end{split}
    \end{equation*}
    for $x, y\in\zinfn$, $z$ and $a_\varepsilon$ as before.

    Let $k = (k_1,\dots,k_n)\in \mathbb{N}^n$, $k_j>1$, $j=1,\dots,n$. We consider $\alpha_j = k_j/2 - 1$, $j=1,\dots,n$, and $\alpha= (\alpha_1,\dots,\alpha_n)$. We have that
    \begin{equation*}
        \mathbb{H}_{\alpha,\varepsilon}^{(2)}(h)(x) =
        \int_{\zinfn} K_{\alpha,\varepsilon}^{(2)}(x,y)h(y) \prod_{j=1}^n
        y_j^{k_j-1} dy \prod_{j=1}^n x_j^{\frac{k_j-1}{p(x)}},
    \end{equation*}
    for $x = (x_1,\dots,x_n)\in\zinfn$, and
    \begin{equation*}
    \begin{split}
        K_{\alpha,\varepsilon}^{(2)}(x,y) & =
        \int_{(-1,1)^n} \chi_{F_x}(y,s)(1-\varphi(x,y,s)) q_+(x,y,s)^{\widehat{k}/2}
        \\ & \quad \times
        \exp\left(-a_\varepsilon  \sqrt{q_+(x,y,s)q_-(x,y,s)}
        \right) \prod_{j=1}^n (1-s_j^2)^{\alpha_j - 1/2} ds,
    \end{split}
    \end{equation*}
    for $x, y\in\zinfn$. We define $\overline{p}(\overline{x_1},\dots,\overline{x_n}) = p(x_1,\dots,x_n)$, where $x_j = |\overline{x_j}|$, $\overline{x_j}\in {\zinf^{k_j}}$, $j=1,\dots,n$. Integrating in multi-radial polar coordinates we have that
    \begin{equation*}
        \left|\mathbb{H}_{\alpha,\varepsilon}^{(2)}(h)(x)\right|
        \leq C
        \int_{|\overline{x}-\overline{y}|>\frac{C_0}{1+|\overline{x}| + \overline{y}|}}
        |\overline{x}+\overline{y}|^{{\widehat{k}}}
        e^{-a_\varepsilon|\overline{x}-\overline{y}||\overline{x}+\overline{y}|}
        |h(|\overline{y_1}|,\dots,|\overline{y_n}|)| d\overline{y} \prod_{j=1}^n x_j^{\frac{k_j-1}{p(x)}},
    \end{equation*}
    for $x = (x_1,\dots,x_n){=(|\overline{x_1}|, \dots, |\overline{x_n}|)}\in\zinfn$ and $\overline{x} = (\overline{x_1},\dots,\overline{x_n})\in \prod_{j=1}^n \mathbb R^{k_j} = \mathbb R^{{\widehat{k}}}$.

    We now proceed as in the above one-dimensional case. In order to do this, {notice that if we define $\overline{p}$ by $\overline{p}(\overline{x}) = p(|\overline{x_1}|,\dots,|\overline{x_n}|)$, for $\overline{x} = (\overline{x_1},\dots,\overline{x_n})\in \zinf^{\widehat{k}}$, then $\overline{p}$ belongs to $\textup{LH}(\zinf^{\widehat{k}})$, with $1<\overline{p}^-\leq \overline{p}^+<\infty$, by virtue of Lemma~\ref{lem:pbarra}. Hence, the Hardy--Littlewood maximal operator $\mhl$ on $\zinf^{\widehat{k}}$ is bounded on $L^{\overline{p}(\cdot)}\left(\zinf^{\widehat{k}}\right)$.}

    We consider, for every $z = (z_1,\dots,z_n)\in \mathbb{C}^n$ such that $\Real(z_j)>-\frac12$, for each ${j=1,\dots,n}$, the operator
    \begin{equation*}
        \mathcal{C}_{z,\varepsilon}(f)(x) =
        \mathbb{H}_{z,\varepsilon}^{(2)}(f_z)(x),\quad x\in\zinfn,
    \end{equation*}
    where $f_z(y)=f(y) \prod_{j=1}^n y_j^{-\frac{2z_j+1}{p(y)}}$ for $y=(y_1,\dots,y_n)\in \zinfn$.
    The proof can be concluded as in the one-dimensional case by using an $n$-dimensional version of the Stein complex interpolation with variable exponent.
    This result can be proved by proceeding as in the proof of \cite[Theorem~1]{MRA} and by using an $n$-dimensional version of the Three Lines Theorem (see Theorem~\ref{3lt} and \cite[Proposition~21]{An}).
    \end{proof}

    \section{Proof of Theorem~\ref{mainthm} for maximal operators}\label{sec4: max op}

    According to the subordination formula \eqref{Ptalfa}, since $\frac{t}{2\sqrt{\pi}}\int_0^\infty \frac{e^{-t^2/(4u)}}{u^{3/2}}du=1$ for each $t>0$, we deduce that
    \[P_*^\alpha(f)(x)\leq W_*^\alpha(f)(x), \quad x\in \zinfn.\]
    Hence, it suffices to see that $W_*^\alpha$ is bounded on $L^{p(\cdot)}(\zinfn,\mu_\alpha)$.

    We firstly study its global part $W_{*,\glob}^\alpha$ given, for $x\in \zinfn$, by 
    \begin{align*}
        &W_{*,\glob}^\alpha(f)(x)\\
        &=\sup_{t>0}\left| \int_{\zinfn}\int_{(-1,1)^n} \frac{e^{-\frac{q_-\left(e^{-t/2}x,y,s\right)}{1-e^{-t}}+|y|^2}}{(1-e^{-t})^{n+\widehat{\alpha}}} (1-\varphi(x,y,s))\Pi_\alpha(s) ds f(y)d\mu_\alpha(y)\right|.
    \end{align*}

    By performing the change of variables $1-e^{-t}=u$, $t>0$, and then replacing $u$ by $t$, we can write
    \begin{align*}
        &W_{*,\glob}^\alpha(f)(x)\\
        &=\sup_{0<t<1}\left| \int_{\zinfn}\int_{(-1,1)^n} \frac{e^{-\frac{q_-\left(\sqrt{1-t}x,y,s\right)}{t}+|y|^2}}{t^{n+\widehat{\alpha}}} (1-\varphi(x,y,s))\Pi_\alpha(s) ds f(y)d\mu_\alpha(y)\right|.
    \end{align*}

    Let $(x,y,s)\in G_1$ (recall the definition on page~\pageref{Gtau}). We consider
    \[u(t)=\frac{(1-t)|x|^2+|y|^2-2\sum_{i=1}^n x_i y_i s_i \sqrt{1-t}}{t}, \quad t\in (0,1).\]
    Setting $a=|x|^2+|y|^2$ and $b=2\sum_{i=1}^n x_i y_i s_i$, we have
     \begin{equation}\label{ut}
      u(t)=\frac at-\frac{\sqrt{1-t}}{t} b-|x|^2,\quad t\in (0,1).
    \end{equation}
    We also define
    \begin{equation*}
      v(t)=\frac{e^{-u(t)}}{t^{n+\widehat{\alpha}}}, \quad t\in (0,1).
    \end{equation*}
    We are going to study the supremum of $v(t)$, for $t\in (0,1)$, by proceeding as in the proof of \cite[Proposition~2.1]{MPS}. The derivative of $v$ is
    \[v'(t)=-\frac{e^{-u(t)}}{t^{n+\widehat{\alpha}}}\left(u'(t)+\frac{n+\widehat{\alpha}}{t}\right), \quad t\in(0,1),\]
    where
    \[u'(t)=-\frac{a}{t^2}+b\left(\frac{1}{2t\sqrt{1-t}}+\frac{\sqrt{1-t}}{t^2}\right)=\frac{-2a\sqrt{1-t}+bt+2b(1-t)}{2t^2\sqrt{1-t}}, \quad t\in (0,1).\]
    Thus,
    \begin{align*}
        v'(t)&=-\frac{e^{-u(t)}}{t^{n+\widehat{\alpha}}}\left(\frac{-2a\sqrt{1-t}-bt+2b}{2t^2\sqrt{1-t}}+\frac{n+\widehat{\alpha}}{t}\right)\\
        &=-\frac{e^{-u(t)}}{t^{n+\widehat{\alpha}}}\cdot \frac{2\sqrt{1-t}(t(n+\widehat{\alpha})-a)+b(2-t)}{2t^2\sqrt{1-t}}, \quad t\in (0,1).
    \end{align*}
    By choosing $C_0>1$ large enough, we can prove $a>n+\widehat{\alpha}$ for any $(x,y,s)\in G_1$. Indeed, let us remark that \[|b|\le 2 \sum_{i=1}^n |x_i||y_i|\le |x|^2+|y|^2=a.\]
Besides, \[a=\frac{a-b+a+b}{2}\ge \frac{q_- (x,y,s)+a-|b|}{2}\ge \frac12 q_-(x,y,s).\] 
Also, \[\sqrt{a}\ge \frac{1}{\sqrt{2}}(|x|+|y|).\] 

Fix $(x,y,s)\in G_1$. If $|x|+|y|<1$ then \[a\ge \frac{1}{2}q_-(x,y,s)>\frac{1}{2}\frac{C_0^2}{(1+|x|+|y|)^2}\ge \frac{C_0^2}{8}>\frac{C_0}{8}\] since we shall take $C_0>1.$ And, if $|x|+|y|\ge 1,$ then \[a=\sqrt{a}\sqrt{a}\ge \frac{1}{\sqrt{2}}(|x|+|y|)\frac{1}{\sqrt{2}}\sqrt{q_-(x,y,s)}>\frac{C_0}{2}\frac{|x|+|y|}{1+|x|+|y|}\ge \frac{C_0}{2}\frac{1}{2}=\frac{C_0}{4}>\frac{C_0}{8}.\]
Therefore, taking $C_0>8 (n+\widehat{\alpha})$ we get that $a>n+\widehat{\alpha}$ on $G_1$ as claimed.

    Then, if $b\leq 0$, $v'(t)>0$ for each $t\in (0,1)$, so
    \[\sup_{0<t<1}v(t)\leq v(1)=e^{-|y|^2}.\]

    On the other hand, if $b>0$, from the property $a>n+\widehat{\alpha}$, the equation
    \[2\sqrt{1-t}(a-t(n+\widehat{\alpha}))=b(2-t)\]
    has a unique solution $t_n$. The arguments developed in \cite[p.~850]{MPS} allow us to conclude that
    \[\sup_{0<t<1} v(t)\sim v(t_0),\]
    where $t_0=2\frac{\sqrt{a^2-b^2}}{a+\sqrt{a^2-b^2}}\sim \sqrt{\frac{q_-(x,y,s)}{q_+(x,y,s)}}$. 
    
    Then,
    \[\sup_{0<t<1} v(t)\leq C\left(\frac{q_+(x,y,s)}{q_-(x,y,s)}\right)^{\frac{n+\widehat{\alpha}}{2}}\exp\left(-\frac{|y|^2-|x|^2}{2}-\frac{{\sqrt{q_+(x,y,s)q_-(x,y,s)}}}{2}\right),\]
    provided that $C_0$ satisfies the above condition. From now on, $C_0$ will be fixed such that the stated condition holds.

    Since $q_+(x,y,s)q_-(x,y,s)\geq c$ for every $(x,y,s)\in G_1$ (see \cite[p.~264]{FSS1})
    , we have that \[\sqrt{\frac{q_+(x,y,s)}{q_-(x,y,s)}}\leq C q_+(x,y,s).\] Therefore, for every $(x,y,s)\in G_1$
    \[\sup_{0<t<1}v(t)\leq C q_+(x,y,s)^{n+\widehat{\alpha}}e^{-\frac{|y|^2-|x|^2}{2}-\frac{{\sqrt{q_+(x,y,s)q_-(x,y,s)}}}{2}}= CH_{\alpha,0}(x,y,s)\]
    where $H_{\alpha,0}$ is the function given in \eqref{H}. Hence, $W^\alpha_{*,\glob}$ is pointwise smaller than a multiple of $\mathcal{H}_{\alpha,0}$, which is a bounded operator on $L^{p(\cdot)}(\zinfn, \mu_\alpha)$ by Proposition~\ref{propoaux}, so $W^\alpha_{*,\glob}$ is also bounded on $L^{p(\cdot)}(\zinfn, \mu_\alpha)$.

    We now study $W^\alpha_{*,\loc}$ defined by 
    \begin{equation*}
        \begin{split}
            &W^\alpha_{*,\loc} (f)(x)\\
            &\quad =\sup_{t>0} \left| \int_{\zinfn} \int_{(-1,1)^n}
            \frac{\exp\left(\frac{-q_-\left(e^{-t/2}x,y,s\right)}{1-e^{-t}}+|y|^2\right)}{(1-e^{-t})^{n+\widehat{\alpha}}}
            \varphi(x,y,s) \Pi_\alpha(s)ds f(y)d\mu_\alpha(y)
            \right|,
        \end{split}
    \end{equation*}
    for $x\in\zinfn$. Setting $u= 1 - e^{-t}$ and then replacing $u$ by $t$, we can write
    \[W^\alpha_{*,\loc} (f)(x)=\sup_{0<t<1} \left| \int_{\zinfn} K_t^\alpha (x,y) f(y)d\mathfrak{m}_\alpha(y)\right|\]
    for $x=(x_1,\dots,x_n)\in\zinfn$, where
    \begin{equation*}
            K^\alpha_t(x,y)
            =
             \int_{(-1,1)^n}
          \frac{\exp\left(-\frac{(1-t)|x|^2+|y|^2 - 2\sqrt{1-t}\sum_{i=1}^n x_iy_is_i}{t}\right)}{t^{n+\widehat{\alpha}}}
            \varphi(x,y,s) \Pi_\alpha(s)ds,
    \end{equation*}
    for $x=(x_1,\dots,x_n)$, $y=(y_1,\dots,y_n)\in\zinfn$ and $t\in(0,1)$. 

    As it was explained in Section~\ref{sec2: method}, we shall see that $W^\alpha_{*,\loc}$ is a bounded operator on $L^{p(\cdot)}(\zinfn,\mu_\alpha)$ as a consequence of vector valued Calder\'on--Zygmund theory.

    According to \cite[(2.6)]{Sa4} we have that 
    \[q_-\left(\sqrt{1-t}x,y,s\right) \geq q_-(x,y,s) - C\left(1-\sqrt{1-t}\right) = q_-(x,y,s) - C \frac{t}{1+\sqrt{1-t}},\]
    for $x, y\in\zinfn$, $t\in(0,1)$, $s\in(-1,1)^n$ and $(x,y,s)\in L_2$.
    
    Then,
    \[|K^\alpha_t(x,y)|
        \leq C
        \int_{(-1,1)^n} \frac{e^{-q_-(x,y,s)/t}}{t^{n+\widehat{\alpha}}}
        \Pi_\alpha(s)ds\leq C
        \int_{(-1,1)^n} \frac{\Pi_\alpha(s)}{q_-(x,y,s)^{n+\widehat{\alpha}}}ds\]
    for $x, y\in\zinfn$ and $t\in(0,1)$.

    According to \cite[Lemma~3.1]{BCN} (see also \cite[Lemma~2.1]{NS}), we get
    \begin{equation}\label{C1}
    \sup_{t>0} |K^\alpha_t(x,y)| \leq
    \frac{C}{\mathfrak{m}_\alpha(B(x,|y-x|))},
    \quad x,y\in\zinfn,
    \;x\neq y.
    \end{equation}

    Let $j=1,\dots,n$. We have that, 
    \begin{equation*}
        \begin{split}
            \partial_{x_j} K^\alpha_t(x,y)
            & =
            \int_{(-1,1)^n}\left( \frac{-2x_j(1-t)+2y_js_j\sqrt{1-t}}{t^{n+1+\widehat{\alpha}}} \varphi(x,y,s) +
            \frac{\frac{\partial \varphi}{\partial x_j}(x,y,s)}{t^{n+\widehat{\alpha}}}
            \right)
            \\ & \qquad
            \times
            \exp\left(-\frac{(1-t)|x|^2+|y|^2-2\sum_{i=1}^n x_i y_i s_i \sqrt{1-t}}{t}\right)
            \Pi_\alpha(s)ds,
        \end{split}
    \end{equation*}
    for $x=(x_1,\dots,x_n)$, $y = (y_1,\dots,y_n)\in\zinfn$, and $t>0$.

    According to the properties of $\varphi$ and using again \cite[Lemma~3.1]{BCN}, since
    \begin{equation}\label{ineq_modq-}
        \begin{split}
            \left|x_j\sqrt{1-t} - y_js_j\right|^2
            & =
            x_j^2(1-t)+ y_j^2 s_j^2 - 2x_j y_j s_j \sqrt{1-t}
            \\ & \leq
            (1-t)|x|^2+ |y|^2 - 2\sum_{i=1}^n x_i y_i s_i \sqrt{1-t},
        \end{split}
    \end{equation}
    for $x=(x_1,\dots,x_n)$, $y = (y_1,\dots,y_n)\in\zinfn$, $t\in (0,1)$ and $s\in (-1,1)^n$, we get
    \begin{equation*}
        \begin{split}
            |\partial_{x_j} K^\alpha_t(x,y)|
            & \leq C
            \int_{(-1,1)^n} \frac{e^{-c\frac{q_-(x,y,s)}{t}}}{t^{n+\frac{1}{2}+\widehat{\alpha}}}
            \Pi_\alpha(s)ds
            \\ & \leq C
            \int_{(-1,1)^n} \frac{\Pi_\alpha(s)}{q_-(x,y,s)^{n+\frac{1}{2}+\widehat{\alpha}}}ds
            \\ & \leq C
            \frac{1}{|x-y|\mathfrak{m}_\alpha(B(x,|x-y|))},
        \end{split}
    \end{equation*}
        for $x, y\in\zinfn$, $x\neq y$, and $t>0$. Hence,
        \begin{equation}\label{C2}
            \sup_{t>0}|\partial_{x_j} K^\alpha_t(x,y)| +
            \sup_{t>0}|\partial_{y_j} K^\alpha_t(x,y)|
            \leq C
            \frac{1}{|x-y|\mathfrak{m}_\alpha(B(x,|x-y|))},
        \end{equation}
        for $x, y\in\zinfn$, $x\neq y$.

        Let $N\in\mathbb{N}$. We consider the space $C\left(\left[1/N,N\right]\right)$ of continuous functions in $\left[1/N,N\right]$ with the usual maximum norm. We define, for every $x, y\in \zinfn$, $x\neq y$,
        \begin{equation*}
            \left[K^\alpha(x,y)\right](t) = K^\alpha_t(x,y), \quad t>0.
        \end{equation*}
        By proceeding as above we can see that, for every $x\in \zinfn$, the mapping $\Phi_x(y)=K^\alpha(x,y)$, $y\in \zinfn$, is continuous from $\zinfn$ into $C\left(\left[1/N,N\right]\right)$, and then, $\Phi_x$ is weakly measurable.
        Since $C\left(\left[1/N,N\right]\right)$ is separable, we conclude that, for every $x\in \zinfn$, $\Phi_x$ is strongly measurable (see \cite[p.~131]{Yos}). According to~\eqref{C1} and~\eqref{C2} we deduce that $K^\alpha$ is a $C\left(\left[1/N,N\right]\right)$-valued Calder\'on--Zygmund kernel with respect to $(\zinfn,|\cdot|,\mathfrak{m}_\alpha)$.

        Suppose $\lambda$ is a complex measure supported in $\left[1/N,N\right]$ and $f\in C^\infty_c(\zinfn)$.
        By using~\eqref{C1} we obtain
        \begin{equation*}
            \int_{[1/N,N]} \int_{\zinfn}
            |K^\alpha_t(x,y)| |f(y)| d\mathfrak{m}_\alpha(y) d|\lambda|(t)<\infty,
            \quad x\notin \supp(f),
        \end{equation*}
        because $|\lambda|([1/N,N])<\infty$. Here $|\lambda|$ denotes the total variation of $\lambda$. It follows that
        \begin{equation}\label{C3}
        \begin{split}
            \int_{[1/N,N]} &\int_{\zinfn}
            K^\alpha_t(x,y) f(y) d\mathfrak{m}_\alpha(y) d\lambda(t)
            \\ & =
            \int_{\zinfn}  \int_{[1/N,N]}
            K^\alpha_t(x,y) f(y)
            d\lambda(t)
            d\mathfrak{m}_\alpha(y) ,
            \quad x\notin \supp(f).
        \end{split}
        \end{equation}

        We define the functional $S_\lambda$ on $C([1/N,N])$ by
        \begin{equation*}
            S_\lambda(g) = \int_{[1/N,N]} g(t)
            d\lambda(t), \quad g\in C([1/N,N]).
        \end{equation*}
        Equality~\eqref{C3} says that, by understanding the integral under $S_\lambda$ in the $C([1/N,N])$-Bochner sense,
        \begin{equation*}
            S_\lambda \left[
            \int_{\zinfn}
            [K^\alpha(x,y)](\cdot) f(y) d\mathfrak{m}_\alpha(y)
            \right]
            =
            \int_{[1/N,N]}  W^\alpha_{t,\loc} (f)(x)
            d\lambda(t),\quad x\notin\supp(f).
        \end{equation*}
        Since the dual of $C([1/N,N])$ is the space $\mathcal{M}([1/N,N])$ of complex measures supported on $[1/N,N]$ we conclude that, for every $x\notin \supp(f)$
        \begin{equation*}
             W^\alpha_{t,\loc} (f)(x) = \left[
             \int_{\zinfn}
            [K^\alpha(x,y)](\cdot) f(y) d\mathfrak{m}_\alpha(y)
             \right](t), \quad t\in [1/N,N].
        \end{equation*}

        According to \cite[p.~73]{StLP}, the maximal operator $W^\alpha_{*}$ is bounded on $L^2(\zinfn,\mu_\alpha)$.
        Also, $W^\alpha_{*,\glob}$ is bounded on $L^2(\zinfn,\mu_\alpha)$ (see the first part of this proof). Then, $W^\alpha_{*,\loc}$ is bounded on $L^2(\zinfn,\mu_\alpha)$.
        Hence, there exists $C>0$ such that, for every $N\in\mathbb{N}$,
        \begin{equation}\label{C4}
            \left\| \left\|
            W^\alpha_{t,\loc}(f)
            \right\|_{C([1/N,N])}
            \right\|_{L^2(\zinfn,\mu_\alpha)}
            \leq C \|f\|_{L^2(\zinfn,\mu_\alpha)},
        \end{equation}
        for $f\in L^2(\zinfn,\mu_\alpha)$. By using~\eqref{C1},~\eqref{C2} and~\eqref{C4} as it was explained in Section~\ref{sec2: method} we get
        \begin{equation*}
            \left\|
            \sup_{t\in[1/N,N]} |W^\alpha_{t,\loc}(f)(x)|
            \right\|_{L^{p(\cdot)}(\zinfn,\mu_\alpha)}
            \leq C
            \|f\|_{L^{p(\cdot)}(\zinfn,\mu_\alpha)}
        \end{equation*}
        for $f\in L^{p(\cdot)}(\zinfn,\mu_\alpha)$ and $C>0$ independent of $N\in\mathbb{N}$.

        By using now the monotone convergence theorem (see \cite[p.~75]{DHHR}), we conclude that $W^\alpha_{*,\loc}$ is bounded on $L^{p(\cdot)}(\zinfn,\mu_\alpha)$. Thus, the proof of Theorem~\ref{mainthm} for
        $W^ \alpha_{*}$ is finished.

      \section{Proof of Theorem~\ref{mainthm} for Riesz transforms}\label{sec5: riesz}

    The proof of Theorem~\ref{mainthm} for Riesz transforms $R^{\beta}_\alpha$ of order $\beta\in \mathbb{N}^n\setminus\{(0,\dots,0)\}$, follows the same steps done in the proof of the results in Section~\ref{sec4: max op} by using some results developed in \cite{FSS1} and \cite{Sa1}. We now sketch the proof.

    Let $\beta\in \mathbb{N}^n\setminus\{(0,\dots,0)\}$ be given. For every $f\in L^2(\zinfn,\mu_\alpha)$, we have that
    \begin{equation*}
        R^{\beta}_\alpha(f)(x) = c_\beta f(x)+\text{p.v.} \int_{\zinfn}
        R^{\beta}_\alpha (x,y) f(y) d\mathfrak{m}_\alpha(y), \quad
        \text{a.e. }x\in\zinfn
    \end{equation*}
    where $c_\beta\in \mathbb{R}$ and
    \begin{equation*}
        R^{\beta}_\alpha (x,y) =\frac{1}{\Gamma\left(\frac{\widehat{\beta}}{2}\right)}
        \int_{(-1,1)^n} K^{\beta}_\alpha (x,y,s) \Pi_\alpha (s)\, d s , \quad
        x, y \in \zinfn, \; x\neq y,
    \end{equation*}
    with \begin{align*}
        K^{\beta}_\alpha(x,y,s)&=\int_0^1 r^{\frac{\widehat{{\beta}}-2}{2}}\left(\frac{-\log r}{1-r}\right)^{\frac{\widehat{{\beta}}-2}{2}} \prod_{i=1}^n H_{{\beta}_i}\left(\frac{\sqrt{r}x_i-y_i s_i}{\sqrt{1-r}}\right)\frac{e^{-\frac{q_-(\sqrt{r} x,y,s)}{1-r}}}{(1-r)^{n+\widehat{\alpha}+1}}\, d r\\ &= \int_0^1  (1-t)^{\frac{\widehat{\beta}-1}{2}}\left(\frac{-\log (1-t)}{t}\right)^{\frac{\widehat{\beta}-2}{2}} \prod_{i=1}^n H_{{\beta}_i}\left(\frac{\sqrt{1-t}x_i-y_i s_i}{\sqrt{t}}\right)\\
        & \ \ \quad \quad \times \frac{e^{-\frac{q_-(\sqrt{1-t}x,y,s)}{t}}}
        {t^{n+\widehat{\alpha}+1}}\, \frac{d t}{\sqrt{1-t}},
    \end{align*}
    being $H_{{\beta}_i}$ the one-dimensional Hermite polynomial of degree ${\beta}_i,$ $i=1, \dots, n$, and for the second equality we have made the change of variables $t=1-r.$ In order to establish that $R_\alpha^\beta$ is bounded on  $L^{p(\cdot)}(\zinfn,\mu_\alpha)$ we can assume that $c_\beta=0$.

    We define $R^{\beta}_{\alpha,\loc}$ and $R^{\beta}_{\alpha,\glob}$ in the usual way. Firstly, we shall prove the $L^{p(\cdot)}$-boundedness of the global part.

    Taking into account that $|\sqrt{1-t}x_i-y_is_i|\le q_-^{\frac{1}{2}}(\sqrt{1-t}x, y, s)$ from \eqref{ineq_modq-}, we get, for every $\varepsilon>0$,
\[\left|\prod_{i=1}^n H_{{\beta}_i}\left(\frac{\sqrt{1-t}x_i-y_is_i}{\sqrt{t}}\right)\right|\leq C \sum_{k=0}^{\widehat{{\beta}}}\left(\frac{q_-^{\frac{1}{2}}(\sqrt{1-t}x,y,s)}{\sqrt{t}}\right)^k\leq C e^{\varepsilon \frac{q_-(\sqrt{1-t}x,y,x)}{t}}. \]

    Also, since the function $t\mapsto (1-t)^{\frac{\widehat{\beta}-1}{2}}\left(-\frac{\log (1-t)}{t}\right)^{\frac{\widehat{\beta}-2}{2}}$ is bounded on $[0,1]$, we have
    \begin{align*}
    \left|R^{\beta}_{\alpha, \glob} f(x)\right|&\le C|f(x) + C \int_{\zinfn} |f(y)| \int_{(-1,1)^n} K_\alpha(x,y,s)\, \Pi_\alpha (s)\, ds\, d\mathfrak{m}_\alpha(y),
    \end{align*}
    for $x\in \zinfn$, being
    \begin{equation*}
        K_\alpha(x,y,s)=\int_0^1 \frac{e^{-(1-\varepsilon)\frac{q_-(\sqrt{1-t}x,y,s)}{t}}}
        {t^{n+\widehat{\alpha}+1}}\, \frac{d t}{\sqrt{1-t}} (1-\varphi(x,y,s))
    \end{equation*}
for $y\in \zinfn$ and $s\in (-1,1)^n.$

We can see that the above kernel is, in turn, bounded by the kernel $H_{\alpha,\varepsilon} (x,y,s)$ given in \eqref{H} provided that $\varepsilon< \frac{1}{n+\widehat{\alpha}}$. When $\displaystyle \sum_{i=1}^n x_i y_i s_i>0 $ we follow closely the estimates obtained by S. P\'erez in \cite{P}, taking into account that in this case, for $0<\varepsilon <\frac{1}{n+\widehat{\alpha}},$
\[K_\alpha (x,y,s)\leq C_\varepsilon \frac{e^{-(1-\varepsilon)u_0}}{t_0^{n+\widehat{\alpha}}}\] with $u_0=\frac{|y|^2-|x|^2+\sqrt{q_+(x,y,s)q_-(x,y,s)}}{2}$ and $t_0=2\frac{\sqrt{a^2-b^2}}{a+\sqrt{a^2-b^2}}$, being $a=|x|^2+|y|^2$ and $b=2\sum_{i=1}^n x_i y_i s_i$.

Indeed, by calling $u(t)=\frac{q_-(\sqrt{1-t}x,y,s)}{t}$, notice that $u$ is the one given in \eqref{ut} at the previous section. We have already proved that, for $b>0$,
\[\sup_{0<t<1}\frac{e^{-u(t)}}{t^{n+\widehat{\alpha}}}\sim
\frac{e^{-u_0}}{t_0^{n+\widehat{\alpha}}}.\] Thus, for $\nu=\frac{1}{n+\widehat{\alpha}}-\varepsilon>0$ we have
\begin{align*}
K_\alpha(x,y,s)&= \int_0^1 e^{\varepsilon u(t)}    \left(\frac{e^{-u(t)}}{t^{n+\widehat{\alpha}}}\right)^{\frac{n+\widehat{\alpha}-1}{n+\widehat{\alpha}}}\left(\frac{e^{-u(t)}}{t^{n+\widehat{\alpha}}}\right)^{\frac{1}{n+\widehat{\alpha}}}\frac{dt}{t\sqrt{1-t}}\\ & \leq C \left(\frac{e^{-u_0}}{t_0^{n+\widehat{\alpha}}}\right)^{1-\frac{1}{n+\widehat{\alpha}}}\int_0^1 e^{-\nu u(t)}\frac{dt}{t^2\sqrt{1-t}}.
\end{align*}
By performing the change of variable $s=u(t)-u_0$ and following the calculations made in \cite[p.~499]{P}, the latter expression is bounded by
\[\frac{e^{-\left(1-\frac{1}{n+\widehat{\alpha}}\right)u_0} e^{-\nu u_0}}{t_0^{n+\widehat{\alpha}-1}}\frac{1}{t_0 \sqrt[4]{(a-b)(a+b)}}\int_0^\infty e^{-\nu s}\left(1+\frac{1}{\sqrt{s}}\right)\, ds.\] Moreover, recalling that $(a-b)(a+b)=q_-(x,y,s) q_+(x,y,s)\ge c$ when $b>0$ (see \cite[p.~264]{FSS1}) we get the estimate claimed above.

For the case $b\le 0,$ we have that $\frac{a}{t}-|x|^2\le u(t)=\frac{q_-(\sqrt{1-t}x,y,s)}{t}$ like in \cite[p.~500]{P}. After making the change of variables $a\left(\frac{1}{t}-1\right)=s$ and performing the integration taking into account that on the global part $a\ge c,$ we get $K_\alpha(x,y,s)\leq Ce^{-(1-\varepsilon)|y|^2}$.

Therefore, $K_\alpha(x,y,s)\leq C H_{\alpha,\varepsilon}(x,y,s)$ for $0<\varepsilon<\frac{1}{n+\widehat{\alpha}}$.
From Proposition~\ref{propoaux} we deduce that the operator $R^{\beta}_{\alpha,\glob}$ is bounded on $L^{p(\cdot)}(\zinfn,\mu_\alpha)$ by choosing ${0<\varepsilon <\frac{1}{n+\widehat{\alpha}}\wedge\frac{1}{(p^-)'}}$.

    According to \cite[p.~699]{NS} $R_\alpha^\beta$ is bounded on $L^2(\zinfn,\mu_\alpha)$. Since, as we have just proved $R_{\alpha,\glob}^\beta$ is bounded on $L^2(\zinfn,\mu_\alpha)$, $R_{\alpha,\loc}^\beta$ is also bounded on $L^2(\zinfn,\mu_\alpha)$. By proceeding as in  \cite[Lemma~3.3]{Sa1} and \cite[Lemma~3.1]{BCN} (see also \cite[Proposition~6~and~Lemma~7]{FSS1})  we can see that the integral kernel of $R^{\beta}_{\alpha,\loc}$ is a Calder\'on--Zygmund kernel with respect to $\mathfrak{m}_\alpha$. The procedure developed in Section~\ref{sec2: method} leads to see that $R^{\beta}_{\alpha,\loc}$ is bounded on $L^{p(\cdot)}(\zinfn,\mu_\alpha)$ and with this we finish the proof of this result.

    \section{Proof of Theorem~\ref{mainthm} for Littlewood--Paley functions} 

    In this section we prove Theorem~\ref{mainthm} for Littlewood--Paley functions $g^{\beta,k}_\alpha$, with $k\in \mathbb{N}$ and $\beta\in \mathbb{N}^n$ such that $k+\widehat{\beta}>0$.

    Let $k\in \mathbb{N}$, $k\geq 1$. We recall that $g_\alpha^k=g_\alpha^{\boldsymbol{0},k},$ i.e.
    \begin{equation*}
        g^k_\alpha (f)(x) =
        \left(
        \int_0^\infty \left|t^k \partial^k_t P^\alpha_t (f)(x)\right|^2 \frac{dt}{t}
        \right)^{1/2},
        \quad x\in \zinf^ n,
    \end{equation*}
    where
    \begin{equation*}
         P^\alpha_t (f)(x) =
         \frac{t}{2\sqrt{\pi}}
        \int_0^\infty \frac{e^{-\frac{t^2}{4u}}}{u^{\frac{3}{2}}}
        W^\alpha_{u}(f)(x) du,
        \quad x\in \zinf^ n
        ,  t>0.
    \end{equation*}
    We define
    \begin{equation*}
         P^\alpha_{t,\loc} (f)(x) =
         \frac{t}{2\sqrt{\pi}}
        \int_0^\infty \frac{e^{-\frac{t^2}{4u}}}{u^{\frac{3}{2}}}
        W^\alpha_{u,\loc}(f)(x) du,
        \quad x\in \zinf^ n
        ,  t>0.
    \end{equation*}
    and
    \begin{equation*}
         P^\alpha_{t,\glob} (f)(x) =
         \frac{t}{2\sqrt{\pi}}
        \int_0^\infty \frac{e^{-\frac{t^2}{4u}}}{u^{\frac{3}{2}}}
        W^\alpha_{u,\glob}(f)(x) du,
        \quad x\in \zinf^ n
        ,  t>0.
    \end{equation*}
     and consider
    \begin{equation*}
        g^k_{\alpha,\loc} (f)(x) =
        \left(
        \int_0^\infty\left|t^k \partial^k_t P^\alpha_{t,\loc} (f)(x)\right|^2 \frac{dt}{t}
        \right)^{1/2},
        \quad x\in \zinf^ n,
    \end{equation*}
    and
    \begin{equation*}
        g^k_{\alpha,\glob} (f)(x) =
        \left(
        \int_0^\infty\left|t^k \partial^k_t P^\alpha_{t,\glob} (f)(x)\right|^2 \frac{dt}{t}
        \right)^{1/2},
        \quad x\in \zinf^ n.
    \end{equation*}

    We firstly prove that $g^k_{\alpha,\glob}$ defines a bounded operator on $L^{p(\cdot)}(\zinfn,\mu_\alpha)$.

    By using Minkowski inequality we get
    \begin{equation*}
        g^k_{\alpha,\glob} (f)(x)
        \leq
        \int_{\zinfn} |f(y)|
        \left(
        \int_0^\infty\left|t^k \partial^k_t P^\alpha_{t,\glob}(x,y)\right|^2 \frac{dt}{t}
        \right)^{1/2}
        d\mu_\alpha(y),
    \end{equation*}
    for $x\in\zinfn$, where
    \begin{equation*}
         P^\alpha_{t,\glob} (x,y) =
         \frac{t}{2\sqrt{\pi}}
        \int_0^\infty \frac{e^{-\frac{t^2}{4u}}}{u^{\frac{3}{2}}}
        W^\alpha_{u,\glob}(x,y) du
        \quad x,y\in \zinf^ n
        ,  t>0.
    \end{equation*}
    We have that
    \begin{equation*}
        \begin{split}
            t^k \partial^k_t P^\alpha_{t,\glob} (x,y)
            & =
            t^k \partial^k_t
            \left[ \frac{1}{\sqrt{\pi}}
            \int_0^\infty
            \frac{e^{-v}}{\sqrt{v}}
            W^\alpha_{\frac{t^2}{4v},\glob}(x,y)dv
            \right]
            \\ & =
            t^k \partial^{k-1}_t
            \left[ \frac{1}{\sqrt{\pi}}
            \int_0^\infty
            \frac{e^{-v}}{\sqrt{v}}
           \partial_t W^\alpha_{\frac{t^2}{4v},\glob}(x,y)dv
            \right]
            \\ & = t^k \partial^{k-1}_t
            \left[ \frac{t}{2\sqrt{\pi}}
            \int_0^\infty
            \frac{e^{-v}}{v^{3/2}}
            \left[\partial_z W^\alpha_{z,\glob}(x,y)
            \right]_{z=\frac{t^2}{4v}}
            dv \right]
            \\ & =
            t^k \partial^{k-1}_t
            \left[ \frac{1}{\sqrt{\pi}}
            \int_0^\infty
            \frac{e^{-\frac{t^2}{4z}}}{\sqrt{z}}
            \partial_z W^\alpha_{z,\glob}(x,y)dz
            \right]
            \\ & =
            \frac{1}{\sqrt{\pi}}
            \int_0^\infty
             t^k \partial^{k-1}_t
             \left[ e^{-\frac{t^2}{4z}} \right]
             \partial_z W^\alpha_{z,\glob}(x,y)
             \frac{dz}{\sqrt{z}},
        \end{split}
    \end{equation*}
    for $x, y\in\zinfn$ and $t>0$.

    By using Minkowski inequality and \cite[Lemma~3]{BCCFR}
    we get
    \begin{equation*}
        \begin{split}
            \|t^k \partial^k_t P^\alpha_{t,\glob} (x,y) \|&_{L^2\left(\zinf,\frac{dt}{t}\right)}
            \\ & \leq
            \frac{1}{\sqrt{\pi}}
            \int_0^\infty
            | \partial_z W^\alpha_{z,\glob}(x,y)|
            \left( \int_0^\infty \left|
             t^k \partial^{k-1}_t
             \left[ e^{-\frac{t^2}{4z}} \right]
            \right|^2 \frac{dt}{t}
            \right)^{\frac{1}{2}}
             \frac{dz}{\sqrt{z}}
             \\ & \leq  C  \int_0^\infty
            | \partial_z W^\alpha_{z,\glob}(x,y)|
            \left( \int_0^\infty
             \frac{e^{-c\frac{t^2}{z}}}{z^{k-1}}
              t^{2k-1} dt
            \right)^{\frac{1}{2}}
            \frac{dz}{\sqrt{z}}
            \\ & \leq  C  \int_0^\infty
            | \partial_z W^\alpha_{z,\glob}(x,y)|  dz
            \quad x,\;y\in\zinfn.
        \end{split}
    \end{equation*}
    We recall that
    \begin{equation*}
        W^\alpha_{z,\glob}(x,y) =
        \frac{1}{(1-e^{-z})^{\widehat{\alpha} + n}}
        \int_{(-1,1)^n} e^{-\frac{q_-(e^{-z/2}x,y,s)}{1-e^{-z}}+|y|^2}
        (1-\varphi(x,y,s)) \Pi_\alpha(s) ds,
    \end{equation*}
    for $x, y\in\zinfn$ and $z>0$. Then,
    \begin{equation*}
        \partial_z W^\alpha_{z,\glob}(x,y) = e^{|y|^2}
        \int_{(-1,1)^n}
        \partial_z \left[
        \frac{e^{-\frac{q_-(e^{-z/2}x,y,s)}{1-e^{-z}}}}{(1-e^{-z})^{\widehat{\alpha} + n}}\right]
        (1-\varphi(x,y,s)) \Pi_\alpha(s) ds,
    \end{equation*}
    for $x, y\in\zinfn$ and $z>0$.

    We obtain
    \begin{equation*}
        \begin{split}
            &\left\| t^k \partial^k_t P^\alpha_{t,\glob} (x,y) \right\|_{L^2\left(\zinf,\frac{dt}{t}\right)}
            \\ & \quad\leq
            C e^{|y|^2}
        \int_{(-1,1)^n} \int_0^\infty \left|
        \partial_z \left[
        \frac{e^{-\frac{q_-(e^{-z/2}x,y,s)}{1-e^{-z}}}}{(1-e^{-z})^{\widehat{\alpha} + n}}\right]
        \right| dz
        (1-\varphi(x,y,s)) \Pi_\alpha(s) ds,
        \end{split}
    \end{equation*}
    We have that
    \begin{equation*}
    \partial_z \left[
        \frac{e^{-\frac{q_-(e^{-z/2}x,y,s)}{1-e^{-z}}}}{(1-e^{-z})^{\widehat{\alpha} + n}}\right]
         =
         \frac{e^{-\frac{q_-(e^{-z/2}x,y,s)}{1-e^{-z}}}}{(1-e^{-z})^{\widehat{\alpha} + n}}
         P_{x,y,s}\left(e^{-z/2}\right),
    \end{equation*}
    for $x, y\in\zinfn$ and $s\in(-1,1)^n$, where, for every $x, y\in\zinfn$ and $s\in(-1,1)^n$, $P_{x,y,s}$ is a polynomial whose degree is at most $4$. Then, for every $x, y\in\zinfn$ and $s\in(-1,1)$, the sign of $P_{x,y,s}$ changes at most four times. We obtain
    \begin{equation*}
        \int_0^\infty \left|
        \partial_z \left[
        \frac{e^{-\frac{q_-(e^{-z/2}x,y,s)}{1-e^{-z}}}}{(1-e^{-z})^{\widehat{\alpha} + n}}\right]
        \right| dz
        \leq C \sup_{z\in\zinf}
        \frac{e^{-\frac{q_-(e^{-z/2}x,y,s)}{1-e^{-z}}}}{(1-e^{-z})^{\widehat{\alpha} + n}}=\sup_{0<t<1} \frac{e^{-\frac{q_-(\sqrt{1-t}x,y,s)}{t}}}{t^{n+\widehat{\alpha}}},
    \end{equation*}
    for $x, y\in\zinfn$ and $s\in(-1,1)^n$.

    This estimate allows us to reduce the analysis of the global operator $g^k_{\alpha,\glob}$ to the operator considered when we studied the operator $W^\alpha_{*,\glob}$  in Section~\ref{sec4: max op}. Thus, we conclude that the operator $g^k_{\alpha,\glob}$ is bounded on $L^{p(\cdot)}(\zinfn,\mu_\alpha)$.

     We now study the operator  $g^k_{\alpha,\loc}$. We will use vector valued Calder\'on--Zygmund theory. In order to have the measurability of the Banach valued functions that appear we are going to consider, for every $N\in\mathbb{N}$, $N\geq 1$, the Banach space $B_N = L^2\left((1/N,N),\frac{dt}{t}\right)$ and in the last step we pass to the limit as $N$ goes to infinity instead of working with the Banach space $L^2(\zinf,\frac{dt}{t})$.
     Let $N\in\mathbb{N}$, $N\geq 1$. We define the operator
     \begin{equation*}
         G^k_{\alpha,\loc}(f)(x,t) = t^k \partial^k_t P^\alpha_{t,\loc}(f)(x), \quad
         x\in\zinf ,  t>0.
     \end{equation*}
    The integral kernel of $ G^k_{\alpha,\loc}$ with respect to $d\mathfrak{m}_\alpha$ is the following
    \begin{equation*}
        M^k_{\alpha,\loc}(x,y,t) = t^k \partial^k_t P^{\alpha}_{t,\loc}(x,y) e^{-|y|^2},
        \quad x, y\in \zinfn ,  t>0.
    \end{equation*}
    Since the Poisson semigroup is a Stein symmetric diffusion semigroup, the function $g^k_\alpha$ is bounded on $L^2(\zinfn,\mu_\alpha)$. In the first part of this proof we establish that $g^k_{\alpha,\glob}$ is bounded on $L^2(\zinfn,\mu_\alpha)$. Thus, there exists $C>0$ that does not depend on $N$ such that
    \begin{equation*}
        \|G^k_{\alpha,\loc}(f)\|_{L^2_{B_N}(\zinfn,\mu_\alpha)} \leq C
        \|f\|_{L^2(\zinfn,\mu_\alpha)}.
    \end{equation*}
    By using Minkowski inequality, \cite[Lemma~4]{BCCFR} and \cite[Lemma~3.1]{BCN} (see also \cite[Proposition~6~and~Lemma~7]{FSS1}) as above, we get
    \begin{equation*}
        \begin{split}
            \|M^k_{\alpha,\loc} & (x,y,t)\|_{L^2\left(\zinfn,\frac{dt}{t}\right)}
             \\ & \leq C
            \int_{(-1,1)^n}|\varphi(x,y,s) \Pi_\alpha(s)
            \int_0^\infty \left|
        \partial_z \left[
        \frac{e^{-\frac{q_-(e^{-z/2}x,y,s)}{1-e^{-z}}}}{(1-e^{-z})^{\widehat{\alpha} + n}}\right]
        \right| dz ds
        \\ & \leq C
            \int_{(-1,1)^n}|\varphi(x,y,s) \Pi_\alpha(s)
             \sup_{z\in\zinf}
        \frac{e^{-\frac{q_-(e^{-z/2}x,y,s)}{1-e^{-z}}}}{(1-e^{-z})^{\widehat{\alpha} + n}} dz ds
        \\ & \leq C
            \int_{(-1,1)^n} \frac{\Pi_\alpha(s)}{q_-(x,y,s)^{\widehat{\alpha} + n}}  ds
         \\ & \leq
         \frac{C}{\mathfrak{m}_\alpha(B(x,|y-x|))},
         \quad x,y\in\zinfn,\;x\neq y.
        \end{split}
    \end{equation*}

    Let $j=1,\dots,n$. By proceeding in a similar way we can see that
    \begin{align*}
            \|\partial_{x_j} M^k_{\alpha,\loc} & (x,y,t)\|_{L^2\left(\zinfn,\frac{dt}{t}\right)}
             \\ & \leq C
            \int_{(-1,1)^n}|\varphi(x,y,s) \Pi_\alpha(s)
            \int_0^\infty \left|
        \partial_z \partial_{x_j}\left[
        \frac{e^{-\frac{q_-(e^{-z/2}x,y,s)}{1-e^{-z}}}}{(1-e^{-z})^{\widehat{\alpha} + n}}\right]
        \right| dz ds
        \\ &  \quad + C
            \int_{(-1,1)^n}|\partial_{x_j}\varphi(x,y,s) \Pi_\alpha(s)
            \int_0^\infty \left|
        \partial_z \left[
        \frac{e^{-\frac{q_-(e^{-z/2}x,y,s)}{1-e^{-z}}}}{(1-e^{-z})^{\widehat{\alpha} + n}}\right]
        \right| dz ds
        \\ & \leq C
            \int_{(-1,1)^n}|\varphi(x,y,s) \Pi_\alpha(s)
            \sup_{z\in\zinf} \left|\partial_{x_j} \left[
        \frac{e^{-\frac{q_-(e^{-z/2}x,y,s)}{1-e^{-z}}}}{(1-e^{-z})^{\widehat{\alpha} + n}}\right] \right| dz ds
        \\ & \quad + C
            \int_{(-1,1)^n}|\partial_{x_j}  \varphi(x,y,s) \Pi_\alpha(s)
             \sup_{z\in\zinf} \left|
        \frac{e^{-\frac{q_-(e^{-z/2}x,y,s)}{1-e^{-z}}}}{(1-e^{-z})^{\widehat{\alpha} + n}} \right| dz ds
        \\ & \leq C
            \int_{(-1,1)^n} \frac{\Pi_\alpha(s)}{q_-(x,y,s)^{\widehat{\alpha} + n + 1/2}}  ds
         \\ & \leq
         \frac{C}{|x-y|\mathfrak{m}_\alpha(B(x,|y-x|))},
         \quad x,y\in\zinfn,\;x\neq y.
        \end{align*}

    Hence,
    \begin{equation}\label{Z1}
        \|M^k_{\alpha,\loc}(x,y)\|_{B_N}
        \leq \frac{C}{\mathfrak{m}_\alpha(B(x,|x-y|))}\quad x,y\in\zinfn,\; x\neq y,
    \end{equation}
    and
       \begin{equation*}
       \begin{split}
        \sum_{j=1}^n\left(
        \|\partial_{x_j}M^k_{\alpha,\loc}(x,y)\right.&\|_{B_N} +\left.
        \|\partial_{y_j}M^k_{\alpha,\loc}(x,y)\|_{B_N}
        \right)
        \\ &
        \leq \frac{C}{|x-y|\mathfrak{m}_\alpha(B(x,|x-y|))}\quad x,y\in\zinfn,\; x\neq y,
       \end{split}
    \end{equation*}
    where $C>0$ does not depend on $N$. Suppose that $h\in B_N$ and $g$ is a smooth function with compact support in $\zinfn$. By using~\eqref{Z1} we deduce that
    \begin{equation*}
    \begin{split}
        \int_{1/N}^N h(t) G^k_{\alpha,\loc} (g)(x,t) \frac{dt}{t}
        & =
        \int_{\zinfn} g(y) \int_{1/N}^N t^k\partial^k_t P^\alpha_{t,\loc} (x,y) h(t) \frac{dt}{t} d\mathfrak{m}_\alpha(y)
        \\ & =
        \int_{1/N}^N h(t) \left[
        \int_{\zinfn} K(x,y) g (y)d\mathfrak{m}_\alpha(y)\right](t)\frac{dt}{t},
    \end{split}
    \end{equation*}
    for $x\notin \supp(f)$, where, for every $x, y\in\zinfn$, $x\neq y$,
    \begin{equation*}
        [K(x,y)](t) = t^k\partial^k_t P^\alpha_{t,\loc}(x,y),\quad \text{a.e. } t\in (1/N,N),
    \end{equation*}
    and the integral in the last line is understood in the $B_N$-Bochner sense. Note that, for every $x\in \zinfn$, the function $\Phi_x$ defined by $\Phi_x(y)=K(x,y)g(y)$, $y\in \zinfn$, is strongly measurable from $\zinfn$ into $B_N$. Indeed, let $x\in \zinfn$. Since $\Phi_x$ is continuous, $\Phi_x$ is weakly measurable. By taking into account that $B_N$ is a separable Banach space, Petti's Theorem (\cite[p.~131]{Yos}) allows us to conclude that $\Phi_x$ is strongly measurable.

    Thus, for every $x\notin\supp(f)$,
    \begin{equation*}
        G^k_{\alpha,\loc}(f)(x,t) = \left[
        \int_0^\infty K(x,y)f(y) d\mathfrak{m}_\alpha(y)
        \right](t),
    \end{equation*}
    in $L^2\left((1/N,N),\frac{dt}{t}\right)$.

    The arguments explained in Section~\ref{sec2: method} allow us to conclude that there exists $C>0$ such that, for every $N\in\mathbb{N}$, $N\geq 1$,
    \begin{equation*}
        \left\| \left\|
        G^k_{\alpha,\loc} (f)
        \right\|_{L^2\left((1/N,N),\frac{dt}{t}\right)}
        \right\|_{L^{p(\cdot)}(\zinfn,\mu_\alpha)}
        \leq C
        \left\| f
        \right\|_{L^{p(\cdot)}(\zinfn,\mu_\alpha)},
    \end{equation*}
    for $f\in L^{p(\cdot)}(\zinfn,\mu_\alpha)$,

    By using the monotone convergence theorem (see \cite[p.~75]{DHHR}) we get
    \begin{equation*}
        \left\| g^k_{\alpha,\loc}(f)
        \right\|_{L^{p(\cdot)}(\zinfn,\mu_\alpha)}
        \leq C
        \left\| f
        \right\|_{L^{p(\cdot)}(\zinfn,\mu_\alpha)},
    \end{equation*}
    for $f\in L^{p(\cdot)}(\zinfn,\mu_\alpha)$, and the proof of our result is finished.

    Let us consider now the Littlewood--Paley functions including also spatial derivatives. For $\beta\in \mathbb{N}^n\setminus \{(0,\dots,0)\}$ and $k\in \mathbb{N}$, we consider
    \begin{equation*}
        g^{\beta,k}_\alpha (f)(x) =
        \left( \int_0^\infty\left|t^{k+\widehat{\beta}} \partial^k_t D^\beta_x
        P_t^\alpha (f)(x)\right|^2 \frac{dt}{t}
        \right)^{1/2},\;
        x\in\zinfn.
    \end{equation*}

    We define the local and global part of $g^{\beta,k}_\alpha$ as follows
    \begin{equation*}
        g^{\beta,k}_{\alpha,\loc} (f)(x) =
        \left( \int_0^\infty\left|t^{k+\widehat{\beta}} \partial^k_t
        P^{\alpha,\beta}_{t,\loc} (f)(x)\right|^2 \frac{dt}{t}
        \right)^{1/2},\;
        x\in\zinfn.
    \end{equation*}
    and
    \begin{equation*}
        g^{\beta,k}_{\alpha,\glob} (f)(x) =
        \left( \int_0^\infty\left|t^{k+\widehat{\beta}} \partial^k_t D^\beta_x
        P^{\alpha,\beta}_{t,\glob} (f)(x)\right|^2 \frac{dt}{t}
        \right)^{1/2},\;
        x\in\zinfn.
    \end{equation*}
    where
    \begin{equation*}
        P^{\alpha,\beta}_{t,\loc} (f)(x) =
        \frac{t}{2\sqrt{\pi}}
        \int_0^\infty \frac{e^{-\frac{t^2}{4u}}}{u^{3/2}} W^{\alpha,\beta}_{u,\loc}(f)(x) dx, \; x\in\zinfn  ,  t>0,
    \end{equation*}
    and
    \begin{equation*}
        P^{\alpha,\beta}_{t,\glob} (f)(x) =
        \frac{t}{2\sqrt{\pi}}
        \int_0^\infty \frac{e^{-\frac{t^2}{4u}}}{u^{3/2}} W^{\alpha,\beta}_{u,\glob}(f)(x) dx, \; x\in\zinfn ,\  t>0.
    \end{equation*}
    Here,
    \begin{equation*}
        W^{\alpha,\beta}_u (f)(x) = \int_{\zinfn} D^\beta_x W^{\alpha}_u (x,t) f(y) d\mu_\alpha(y),
        \; x\in\zinfn ,  u>0,
    \end{equation*}
    and $W^{\alpha,\beta}_{u,\loc}$ and $W^{\alpha,\beta}_{u,\glob}$ are defined in the usual way.

    By using Minkowski inequality and~\cite[Lemma~4]{BCCFR} we obtain
    \begin{align*}
       g^{\beta,k}_{\alpha,\glob} (f)(x) &
        \leq C \int_{\zinfn}|f(y)| (1-\varphi(x,y))
        \\
        &\quad \times \left(
        \int_0^\infty \left| t^{k + \widehat{\beta}} \partial^k_t
        \left[\int_0^\infty
        \frac{t e^{-\frac{t^2}{4u}}}{u^{3/2}}
        D^\beta_x W^\alpha_u (x,y) du
        \right]
        \right|^2
        \frac{dt}{t}
        \right)^{1/2} d\mu_\alpha(y)
        \\ &
        \leq C
        \int_{\zinfn}|f(y)| (1-\varphi(x,y))\\
        &\quad \times
        \int_0^\infty
        \left(
        \int_0^\infty |
        t^{k + \widehat{\beta}}
        \partial^k_t
        (t e^{-\frac{t^2}{4u}})| \frac{dt}{t}
        \right)^{1/2}
        |D^\beta_x W^\alpha_u (x,y)|
        \frac{du}{u^{\frac32}}
        d\mu_\alpha(y)
        \\ & \leq C
        \int_{\zinfn}|f(y)| (1-\varphi(x,y))
        \int_0^\infty
        u^{\widehat{\beta}/2 - 1}
        |D^{\beta}_x W_u(x,y)|du
       d\mu_\alpha(y),
    \end{align*}
    for $x\in \zinfn$. From now on we follow the same steps we have done for the higher order Riesz-Laguerre transforms restricted
    to the global part in order to get the $L^{p(\cdot)}$-boundedness of this operator too, taking into account the representation
    given in \eqref{Rieszkernelderiv}.

    In order to study the local operator $g^{\beta,k}_{\alpha,\loc}$ we use the vector valued Calder\'on--Zygmund theory. We consider the operator $G^{\beta,k}_{\alpha,\loc}$ defined by
    \begin{equation*}
        G^{\beta,k}_{\alpha,\loc} (f)(x,t) = t^{k + \widehat{\beta}}
        \partial^k_t P^{\alpha,\beta}_{t,\loc} (f)(x), \;x\in\zinfn ,  t>0.
    \end{equation*}
    The integral kernel $M^{\beta,k}_{\alpha,\loc}$ of the above operator  with respect to $\mathfrak{m}_\alpha$ can be written as follows
    \begin{equation*}
        M^{\beta, k}_{\alpha,\loc} (x,y,t) = \int_{(-1,1)^n} \varphi(x,y) M^{\beta,k}_\alpha (x,y,t,s) \Pi_\alpha(s) ds, \;x,\,y \in\zinfn, ,  t>0,
    \end{equation*}
    where
    \begin{equation*}
        M^{\beta,k}_\alpha(x,y,t,s) = \frac{1}{2\sqrt{\pi}}
        \int_0^\infty \frac{t^{k+\widehat{\beta}} \partial^k_t
        \left[ t e^{-\frac{t^2}{4u}}
        \right]}{u^{3/2}(1-e^{-u})^{n+\widehat{\alpha}}}
        D^\beta_x \left[
        e^{-\frac{q_-(e^{-u/2}x,y,s)}{1-e^{-u + |y|^2}}}
        \right] du,
    \end{equation*}
    for $x, y\in\zinfn$, $t>0$ and $s\in(-1,1)^n$.
    By using Minkowski inequality and~\cite[Lemma~4]{BCCFR}, according to~\cite[(2.3)]{FSS1}, we deduce that, for every $x, y\in\zinfn$ and $s\in(-1,1)^n$,
    \begin{equation*}
        \begin{split}
            & \| M^{\beta,k}_\alpha(x,y,t,s)\|_{L^2\left(\zinf,\frac{dt}{t}\right)}\\
            &\leq C
            \int_0^\infty \frac{\left\|t^{k+\widehat{\beta}} \partial^k_t
        \left[ t e^{-\frac{t^2}{4u}}
        \right]\right\|_{L^2\left(\zinf,\frac{dt}{t}\right)}}{u^{3/2}(1-e^{-u})^{n+\widehat{\alpha}}}
        \left |
        D^\beta_x \left[
        e^{-\frac{q_-(e^{-u/2}x,y,s)}{1-e^{-u + |y|^2}}}
        \right]  \right |   du
        \\ & \leq C
        \int_0^\infty
        \frac{u^{\widehat{\beta}/2-1}}{u^{3/2}(1-e^{-u})^{n+\widehat{\alpha}}}
       \left |
        D^\beta_x \left[
        e^{-\frac{q_-(e^{-u/2}x,y,s)}{1-e^{-u + |y|^2}}}
        \right]  \right |   du
        \\ & \leq C
        \int_0^1 \sqrt{r}^{\widehat{\beta}-2}
        \left(-\frac{\log(r)}{1-r}\right)^{\frac{\widehat{\beta}-2}{2}}
        \prod_{i=1}^n \left | H_{\beta_i}
        \left(\frac{\sqrt{r}x_i-y_is_i}{\sqrt{1-r}}\right)\right |
        \frac{e^{-\frac{q_-(rx,y,s)}{1-r}}}{(1-r)^{n+\widehat{\alpha}+1}} dr,
        \end{split}
    \end{equation*}
    where we recall that, for every $j\in\mathbb{N}$, $H_j$ denotes the one-dimensional Hermite polynomial of degree $j$.

    As in the Riesz transform $R_{\alpha,\loc}^\beta$ case (Section~\ref{sec5: riesz}), we obtain that
    \begin{equation*}
        \left\| M^{\beta,k}_{\alpha,\loc}(x,y,\cdot)\right\|_{L^2\left(\zinf,\frac{dt}{t}\right)} \leq \frac{C}{\mathfrak{m}_\alpha (B(x,|x-y|))},
        \;x,\,y\in \zinfn, \, x\neq y.
    \end{equation*}
    In a similar way we can see that
    \begin{align*}
         \sum_{i=1}^n&\left(\left\| \partial_{x_i}M^{\beta,k}_{\alpha,\loc}(x,y,\cdot)\right\|_{L^2\left(\zinf,\frac{dt}{t}\right)}+
         \left\| \partial_{y_i}M^{\beta,k}_{\alpha,\loc}(x,y,\cdot)\right\|_{L^2\left(\zinf,\frac{dt}{t}\right)}\right)\\
         &  \leq \frac{C}{|x-y|\mathfrak{m}_\alpha(B(x,|x-y|))},
    \end{align*}
    for $x, y\in\zinfn$, $x\neq y$.

    By proceeding as in the first part of the proof when $\beta=0$, we can prove that the local operator $g^{\beta,k}_{\alpha,\loc}$ is bounded on $L^{p(\cdot)}(\zinfn,\mu_\alpha)$ when we show that
    $g^{\beta,k}_{\alpha,\loc}$ is bounded on $L^{2}(\zinfn,\mu_\alpha)$.

    We are going to see that $g^{\beta,k}_{\alpha,\loc}$ is bounded on $L^{2}(\zinfn,\mu_\alpha)$ by proving that  $g^{\beta,k}_{\alpha}$ is bounded on $L^{2}(\zinfn,\mu_\alpha)$. Then, since we have proved that  $g^{\beta,k}_{\alpha,\glob}$ is bounded on $L^{2}(\zinfn,\mu_\alpha)$, we conclude that  $g^{\beta,k}_{\alpha,\loc}$ is bounded on $L^{2}(\zinfn,\mu_\alpha)$.

    According to~\cite[p.~699]{NS}, and by performing a change of variables we obtain that
    \begin{equation*}
        D^{\beta}_x \mathcal{L}^\alpha_r(x)
        =
        \sum_{(m,\ell)\in\mathcal{A}(\beta)} C^{\beta,\alpha}_{m,\ell}(r)
        \left(\prod_{i=1}^n x_i^{\beta_i-m_i}\right)
        \mathcal{L}^{\alpha+\beta-m}_{r-\beta+m+\ell}(x),
    \end{equation*}
    for $x=(x_1,\dots,x_n)\in\zinfn$ and $r\in\mathbb{N}^n$, with
    \begin{equation*}
        \mathcal{A}(\beta) =
        \left\{
        (m,\ell)\in \mathbb{N}\times\mathbb{N}^n: 0\leq m_j \leq \beta_j, 0\leq \ell_j\leq \frac{\beta_j-m_j}{2}, j=1,\dots,n
        \right\}.
    \end{equation*}
    Furthermore, for every $(m,\ell)\in \mathcal{A}(\beta)$ and $k\in \mathbb{N}^n$, $C^{\beta,\alpha}_{(m,\ell)}\in \mathbb{R}$ and
    \begin{equation}\label{aux}
        \left|C^{\beta,\alpha}_{(m,\ell)}\right|
        \left\|\left(\prod_{i=1}^n x_i^{\beta_i-m_i}\right)
        \mathcal{L}^{\alpha+\beta-m}_{k-\beta+m+\ell}\right\|_{L^2(\zinf,\mu_\alpha)} \leq
        C_\beta \lambda^{\widehat{\beta}/2}_k.
    \end{equation}

    Suppose that $f=\sum_{r\in\Lambda}w_r \mathcal{L}^\alpha_r$, where $\Lambda$ 
    is a finite subset of $\mathbb{N}^n$ and $w_r\in\mathbb{C}$ for $r\in \Lambda$. Since for every $(m,\ell)\in \mathcal{A}(\beta)$, the system
    \begin{equation*}
        \left\{\left(\prod_{i=1}^n x_i^{\beta_i-m_i}\right)
        \mathcal{L}^{\alpha+\beta-m}_{r-\beta+m+\ell}
        \right\}_{k\in\Lambda_{m,\ell}}
    \end{equation*}
    is orthogonal with respect to $\mu_\alpha$, where $\Lambda_{m,\ell}=\{k\in\mathbb{N}^n: k_j-\beta_j+m_j+\ell_j\geq 0,\, j=1,\dots,n\}$, Bessel inequality leads, by using~\eqref{aux}, to
    \begin{equation*}
        \begin{split}
            \|g^{\beta,k}_\alpha (f)\|^2_{L^2(\zinfn,\mu_\alpha)}&=
            \int_0^\infty t^{2(k+\widehat{\beta})-1}
            \int_{\zinfn}\left|
            \sum_{r\in\Lambda} \lambda^{k/2}_r e^{-\sqrt{\lambda_r}t}
            c_r D^\beta_x \mathcal{L}^\alpha_r(x)
            \right|^2d\mu_\alpha(x) dt
        \\ & \leq C
        \int_0^\infty t^{2(k+\widehat{\beta})-1}
        \sum_{r\in\Lambda}|c_r|^2 e^{-2\sqrt{\lambda_r}t}
        \lambda_r^{k+\widehat{\beta}}dt
        \\ & \leq C
        \sum_{r\in\Lambda}|c_r|^2
        = C \|f\|^2_{L^2(\zinfn,\mu_\alpha)}.
        \end{split}
    \end{equation*}

    Suppose now that $f\in L^2(\zinfn,\mu_\alpha)$. For every $m\in \mathbb{N}$, we define
    \begin{equation*}
        f_m = \sum_{\gamma\in\mathbb{N}^n, \widehat{\gamma}\leq m} c^\alpha_\gamma(f)\mathcal{L}^\alpha_\gamma.
    \end{equation*}
    We have that $f_m\rightarrow f$, as $m\rightarrow \infty$, in $L^2(\zinfn,\mu_\alpha)$.

    It follows that
    \begin{equation*}
    \begin{split}
        |D^\beta_x W^\alpha_t(x,y)| & \leq \frac{C}{(1-e^{-t})^{n+\widehat{\alpha}}}
        \int_{(-1,1)^n} \left|
        \partial^\beta_x e^{-\frac{q_-(e^{-t/2}x, y,s)}{1-e^{-t}}+|y|^2}
        \right|
        \prod_{i=1}^n (1-s_i^2)^{\alpha_i - 1/2} ds
        \\ & \leq C
        \frac{e^{-t/2}}{(1-e^{-t})^r}
        V(|x|,|y|), \quad x,y\in\zinfn,  t>0,
    \end{split}
    \end{equation*}
    where $r\geq n + \widehat{\alpha}$ and $V$ is a polynomial with positive coefficients. By using~\cite[Lemma~4]{BCCFR} we get
    \begin{equation*}
        \begin{split}
           \left|\partial^k_t D_x^\beta P^\alpha_t(x,y)\right| &
            \leq C \int_0^\infty \frac{\left|\partial^k_t
            \left[t e^{-\frac{t^2}{4u}}\right]\right|}{u^{3/2}} |D^\beta_x W^\alpha_u(x,y)| du
            \\ & \leq C
            \int_0^\infty
         \frac{e^{-\frac{t^2}{8u}}e^{-u/2}}{u^{\frac{k+2}{2}}(1-e^{-u})^r} du V(|x|,|y|)
         \\ & \leq C
         \left( 1+ \int_0^1 \frac{e^{-\frac{t^2}{8u}}}{u^{\frac{k+2}{2}+r}} du \right)
         V(|x|,|y|)
         \\ & \leq C
         \left( 1+ t^{-k-1-2r} \right)
         V(|x|,|y|),
         \quad x,y\in\zinfn ,  t>0.
        \end{split}
    \end{equation*}
    Therefore,
    \begin{equation*}
        \begin{split}
           \left|\partial^k_t D_x^\beta P^\alpha_t(f_m-f)(x)\right| &
            \leq C
         \left( 1+ t^{-k-1-2r} \right)
         \int_{\mathbb{R}^n}|f_m(y) - f(y)|
         V(|x|,|y|) d\mu_\alpha(y)
         \\ & \leq C
         \left( 1+ t^{-k-1-2r} \right)
         \left(\int_{\mathbb{R}^n}
         V^2(|x|,|y|) d\mu_\alpha(y)
         \right)^{1/2}\\
         &\quad \times
         \|f_m(y) - f(y)\|_{L^2(\zinfn,\mu_\alpha)},
        \end{split}
    \end{equation*}
    for $x, y\in\zinfn$, $t>0$ and $m\in \mathbb{N}$.

    We deduce that
    \begin{equation*}
        \lim_{m\rightarrow \infty}
        t^{k+\widehat{\beta}} \partial^k_t D_x^\beta P^\alpha_t(f_m)(x)
        = t^{k+\widehat{\beta}} \partial^k_t D_x^\beta P^\alpha_t(f)(x),
    \end{equation*}
    for $x\in \zinfn$ and $t>0$.

    By using Fatou's Lemma twice we get
    \begin{equation*}
        \begin{split}
            g^{\beta,k}_\alpha (f)(x)
            & =
            \left( \int_0^\infty
           \left|t^{k+\widehat{\beta}} \partial^k_t D_x^\beta P^\alpha_t(f)(x)\right|^2 \frac{dt}{t}
            \right)^{1/2}
            \\ & =
            \left( \int_0^\infty
            \lim_{m\rightarrow \infty}\left|t^{k+\widehat{\beta}} \partial^k_t D_x^\beta P^\alpha_t(f_m)(x)\right|^2 \frac{dt}{t}
            \right)^{1/2}
            \\ & \leq
            \liminf_{m\rightarrow \infty}
            g^{\beta,k}_\alpha (f_m)(x),\quad x\in\zinfn,
        \end{split}
    \end{equation*}
    and then
    \begin{equation*}
        \begin{split}
            \|g^{\beta,k}_\alpha (f)\|_{L^2(\zinfn,\mu_\alpha)} & \leq
            \left( \int_{\zinfn}
            \liminf_{m\rightarrow \infty}
            |g^{\beta,k}_\alpha (f_m)(x)|^2 d\mu_\alpha(x)
            \right)^{1/2}
            \\ & \leq  \liminf_{m\rightarrow \infty}
            \|g^{\beta,k}_\alpha (f_m)\|_{L^2(\zinfn,\mu_\alpha)}
            \\ & \leq C \lim_{m\rightarrow \infty}
          \|f_m\|_{L^2(\zinfn,\mu_\alpha)}
            \\ & \leq C
          \|f\|_{L^2(\zinfn,\mu_\alpha)}.
        \end{split}
    \end{equation*}
    Thus, we have proved that $g^{\beta,k}_\alpha$ is bounded on $L^2(\zinfn,\mu_\alpha)$.

    \section{Proof of Theorem~\ref{mainthm} for Laplace transform type multipliers}
    We recall that we have
    \begin{equation*}
        T^\alpha_m (f)(x)
        = \lim_{\epsilon\rightarrow 0^+}\left(f(x)\Lambda(\epsilon)+
        \int_{|x-y|>\epsilon} K^\alpha_\phi (x,y) f(y) d\mu_\alpha(y)\right), \;\;\text{a.e. } x\in\zinfn,
    \end{equation*}
    where $\Lambda\in L^\infty(\zinf)$ and
    \begin{equation*}
        K^\alpha_\phi (x,y) =
        \int_0^\infty \phi(t) \left(-\frac{\partial}{\partial t}\right)W^\alpha_t(x,y)dt,\quad x,y\in\zinfn,\; x\neq y,
    \end{equation*}
    being $\phi\in L^\infty(\zinfn)$ and $m(t) = t\int_0^\infty e^{-zt}\phi(z)dz$, $t\in\zinf$.

    We define $T^\alpha_{m,\loc}$, $K^\alpha_{\phi,\loc}$, $T^\alpha_{m,\glob}$ and $K^\alpha_{\phi,\glob}$ in the usual way.
    We firstly observe that
    \begin{equation*}
        \begin{split}
            |K^\alpha_{\phi,\glob}&(x,y)|\\
            & \leq C
            \int_{(-1,1)^n} \int_0^\infty
            \left|
            \partial_t \left[
            \frac{e^{-\frac{q_-(e^{-t/2}x,y,s)}{1-e^{-t}}}}{(1-e^{-t})^{\widehat{\alpha} + n}}
            \right]
            \right|
            |\phi(t)|dt
            |1-\varphi(x,y,s)| \Pi_\alpha(s)ds
            \\ & \leq C
            \int_{(-1,1)^n}
            \sup_{t>0}\left|
            \frac{e^{-\frac{q_-(e^{-t/2}x,y,s)}{1-e^{-t}}}}{(1-e^{-t})^{\widehat{\alpha} + n}}
            \right|
            |1-\varphi(x,y,s)| \Pi_\alpha(s)ds.
        \end{split}
    \end{equation*}
    By proceeding as in the proof of Section~\ref{sec4: max op} we conclude that $T^\alpha_{m,\glob}$ is bounded on $L^{p(\cdot)}(\zinfn,\mu_\alpha)$.

    We now define $\mathbb{K}^\alpha_{\phi,\loc}(x,y) = e^{-|y|^2} K^\alpha_{\phi,\loc}(x,y)$ for $x, y\in\zinfn$. By using \cite[Lemma~1]{Sa4} and \cite[Lemma~3.1]{BCN} we can see that $\mathbb{K}^\alpha_{\phi,\loc}(x,y)$ is a scalar Calder\'on--Zygmund kernel with respect to $\mathfrak{m}_\alpha$. According to \cite[Corollary~3, p.~121]{StLP}, the Laguerre multiplier $T^{\alpha}_m$ is bounded on $L^2(\zinfn,\mu_\alpha)$. Furthermore, as we have just mentioned $T^\alpha_{m,\glob}$ is bounded on $L^2(\zinfn,\mu_\alpha)$. Then, $T^\alpha_{m,\loc}$ is bounded on $L^2(\zinfn,\mu_\alpha)$ and on $L^2(\zinfn,\mathfrak{m}_\alpha)$.

    As it was proved in Section~\ref{sec2: method}, we can conclude that  $\mathbb{T}^\alpha_{m,\loc}$ is bounded on $L^{p(\cdot)}(\zinfn,\mu_\alpha)$ and finish the proof of our result.

\appendix \section{Auxiliary results}  For the sake of completeness we include in this appendix an $n$-dimensional version of the Three Lines Theorem
in the form it was used in Section~\ref{sec3: aux op}. Although it can be seen as a particular case of \cite[Proposition~21]{An}, we believe that this simpler form might be enough in many circumstances.

\begin{thm}\label{3lt}
Let $n\in\mathbb{N}$, $n\geq 1$. Assume that, for every $j=1,\dots,n$,
$a_j,\,b_j\in \mathbb{R}$ and $a_j<b_j$. We define
$\tau_n=\{z\in\mathbb{C}^n: a_j\leq\Real(z_j)\leq b_j,
j=1,\dots,n\}$ and $\mathcal{F}_n = \{z\in\mathbb{C}^n: \Real(z_j)\in
\{a_j,b_j\}, j=1,\dots,n \}$. Suppose that $U$ is an open set containing $\tau_n$ and $f:U \to\mathbb{C}$ is
holomorphic, bounded
in $\tau_n$,  and such that $|f(z)|\leq K$ for $z\in \mathcal{F}_n$. Then,
$|f(z)|\leq K$ for $z\in \tau_n$.
\end{thm}

\begin{proof}
We will proceed by induction on the dimension $n$. The case $n=1$
corresponds to the classical Three Lines Theorem and we refer
to \cite[Theorem~3.15]{Pal}.

Suppose the result is true for some $n\in\mathbb{N}$, $n\geq 1$. We consider
$a_j<b_j$ for $j=1,\dots,n+1$, $\tau_{n+1} =
\{z=(z_1,\dots,z_{n+1})\in \mathbb{C}^{n+1}: a_j\leq \Real(z_j)\leq
b_j, j=1,\dots,n+1\}$, $\mathcal{F}_{n+1} = \{ z=(z_1,\dots,z_{n+1}) \in
\mathbb{C}^{n+1}: \Real(z_j)\in\{a_j,b_j\}, j=1,\dots, n+1\}$, an open set $U$
containing $\tau_{n+1}$, and a function
$f:U\rightarrow \mathbb{C}$, holomorphic in $U$,  bounded on $\tau_{n+1}$, and such that
$|f(z)|\leq K$ for $z\in \mathcal{F}_{n+1}$.

Let $t\in \mathbb{R}$. We define, $z_{n+1}(t)= a_{n+1} + it$, and
$g_t:U_t \rightarrow \mathbb{C}$ such that $g_t(z_1,\dots,z_n) =
f(z_1,\dots,z_n,z_{n+1}(t))$, where \[U_t=\{z=(z_1,\dots,z_n)\in
\mathbb{C}^n:\,(z_1,\dots,z_n,z_{n+1}(t))\in U\}.\] It is clear that
$U_t$ is an open set in $\mathbb{C}^n$ that contains $\tau_n$. 
The function $g_t$ is holomorphic in $U_t$ and
bounded on $\tau_n$, and if $z= (z_1,\dots,z_n)\in \mathcal{F}_n$, since $\Real(z_{n+1})=a_{n+1}$, $|g_t(z_1, \dots,
z_n)| = |f(z_1,\dots,z_n,z_{n+1})|\leq K$. Then, using the inductive
hypothesis,
\begin{equation*}
|g_t(z_1,\dots,z_n)| = |f(z_1,\dots,z_n, z_{n+1}(t))| \leq K
\end{equation*}
for $z=(z_1,\dots,z_n)\in \tau_n$. Thus, we prove that
\begin{equation}\label{eq-3l1}
|f(z_1,\dots,z_{n+1})| \leq K
\text{ if } \Real(z_j)\in [a_j,b_j],\,
j=1,\dots,n; \ \Real(z_{n+1})=a_{n+1}.
\end{equation}

In a similar way, we can see that
\begin{equation}\label{eq-3l2}
|f(z_1,\dots,z_{n+1})| \leq K
\text{ if } \Real(z_j)\in [a_j,b_j],\,
j=1,\dots,n ; \ \Real(z_{n+1})=b_{n+1}.
\end{equation}

Let now $c=(c_1,\dots,c_n) \in \prod_{j=1}^n [a_j,b_j]$ and
$t=(t_1,\dots,t_n)\in \mathbb{R}^n$. We consider $\tau_0 = \{
z\in\mathbb{C}: \Real(z)\in [a_{n+1}, b_{n+1}]\}$, $ \mathcal{F}_0 = \{ z\in
\mathbb{C}: \Real(z) \in \{a_{n+1}, b_{n+1}\}\}$, and $h^c_t: U_0
\rightarrow \mathbb{C}$ such that
\[
h^c_t(z) = f(c_1 + it_1,\dots,c_n + it_n, z),\quad z\in U_0,
\]
where $U_0=\{z\in \mathbb{C}:\,(c_1 + it_1,\dots,c_n + it_n, z)\in
U\}$. The set $U_0$ is open in $\mathbb{C}$ and it contains
$\tau_0$. The function $h^c_t$ is holomorphic in $U_0$ and bounded on
$\tau_0$. Furthermore, by~\eqref{eq-3l1} and~\eqref{eq-3l2}, if
$z\in \mathcal{F}_0$,
\begin{equation*}
|h^c_t(z)| = |f(c_1 + i t_1, \dots, c_n + i t_n, z)| \leq K.
\end{equation*}

Therefore, by the one-dimensional case, we deduce that
\begin{equation*}
|h^c_t(z)|\leq K,\quad z\in\tau_0.
\end{equation*}

Thus we conclude that
\begin{equation*}
|f(z)| \leq K,\quad z\in \tau_{n+1}.\qedhere
\end{equation*}
\end{proof}

\begin{lem}\label{lem:pbarra} Let $p:\zinfn\rightarrow [1,\infty)$ be a measurable function such that $p\in \textup{LH}(\zinfn)$ and take $k=(k_1,\dots, k_n)\in \mathbb N^n$ with $k_j\ge 1$ for each $j=1,\dots, n$. Consider $\overline{x}=(\overline{x_1},\dots, \overline{x_n})\in \mathbb R^{\widehat{k}}$ with $\overline{x_j}\in \mathbb R^{k_j}$, $j=1,\dots, n$. We define $\overline{p}:\mathbb R^{\widehat{k}} \rightarrow [1,\infty)$ by $\overline{p}(\overline{x})=p(|\overline{x_1}|,\dots, |\overline{x_n}|)$. Then, $\overline{p}\in \textup{LH}(\mathbb R^{\widehat{k}})$. Moreover, if $1<p^-\leq p^+<\infty$, also $1<\overline{p}^-\leq \overline{p}^+<\infty$.
\end{lem}

\begin{proof}First, we shall see that $\overline{p}$ belongs to $\textup{LH}_0(\mathbb R^{\widehat{k}})$, so we take $\overline{x} = (\overline{x_1},\dots,\overline{x_n})$, $\overline{y} = (\overline{y_1},\dots,\overline{y_n}) \in \mathbb R^{\widehat{k}}$, with $\overline{x_j}, \overline{y_j}\in \mathbb R^{k_j}$, $j=1,\dots, n$, and such that $ 0 < |\overline{x} - \overline{y}| < \frac12$. We have that
     \begin{equation*}
       |(|\overline{x_1}|-|\overline{y_1}|,\dots,|\overline{x_n}|-|\overline{y_n}|)| \leq |\overline{x}-\overline{y}|.
     \end{equation*}
     Indeed, if we write $\overline{x_j} = (x_{1}^j,\dots,x_{{k_j}}^j)$, $\overline{y_j} = (y_{1}^j,\dots,y_{{k_j}}^j)$, with $j=1,\dots,n$,
     this inequality is a consequence of the Cauchy--Schwarz inequality on $\mathbb{R}^{k_j},$ i.e.
     $|\langle \overline{x_j},\overline{y_j}\rangle|\le |\overline{x_j}||\overline{y_j}|$, $j=1,\dots,n$.

    Since $p\in \textup{LH}_0(\zinfn)$ it follows that
    \begin{equation*}
    \begin{split}
        |\overline{p}(\overline{x}) - \overline{p}(\overline{y})|
        & =
        |p(|\overline{x_1}|,\dots,|\overline{x_n}|) - p(|\overline{y_1}|,\dots,|\overline{y_n}|)|
        \\ & \leq \frac{C}{-\log(|(|\overline{x_1}|-|\overline{y_1}|,
        \dots,|\overline{x_n}|-|\overline{y_n}|)|)}
        \\ & \leq
        \frac{C}{-\log(|\overline{x}-\overline{y}|)}.
    \end{split}
    \end{equation*}
    Thus, $\overline{p}\in \textup{LH}_0(\mathbb R^{\widehat{k}})$.

    On the other hand, since $p\in \textup{LH}_\infty(\zinfn)$ and $|(|\overline{x_1}|,\dots, |\overline{x_n}|)|=|\overline{x}|$,
    \[|\overline{p}(\overline{x})-p_\infty|=|p(|\overline{x_1}|,\dots, |\overline{x_n}|)-p_\infty|\leq \frac{C}{\log(e+|(|\overline{x_1}|,\dots, |\overline{x_n}|)|)}=\frac{C}{\log(e+|\overline{x}|)},\]
    so $\overline{p}\in\textup{LH}_\infty(\mathbb{R}^{\widehat{k}})$ with $\overline{p}_\infty=p_\infty$.

    Therefore, we have proved that $\overline{p}\in\textup{LH}(\mathbb R^{\widehat{k}}) $.

    Finally, from the definition of $\overline{p}$, it is clear that $\overline{p}^-=p^-$ and $\overline{p}^+=p^+$, so $1<p^-\leq p^+<\infty$ is equivalent to $1<\overline{p}^-\leq \overline{p}^+<\infty$.
\end{proof}

\bibliographystyle{acm}

\bibliography{BDQS}

\begin{thebibliography}{10}

\bibitem{AM}
{\sc Acerbi, E., and Mingione, G.}
\newblock Regularity results for a class of functionals with non-standard
  growth.
\newblock {\em Arch. Ration. Mech. Anal. 156}, 2 (2001), 121--140.

\bibitem{AHH}
{\sc Adamowicz, T., Harjulehto, P., and H\"{a}st\"{o}, P.}
\newblock Maximal operator in variable exponent {L}ebesgue spaces on unbounded
  quasimetric measure spaces.
\newblock {\em Math. Scand. 116}, 1 (2015), 5--22.

\bibitem{AD}
{\sc Anderson, T.~C., and Damián, W.}
\newblock Calder\'on-{Z}ygmund operators and commutators in spaces of
  homogeneous type: weighted inequalities, 2014.
\newblock arXiv:1401.2061.

\bibitem{An}
{\sc Anker, J.-P.}
\newblock {${\bf L}_p$} {F}ourier multipliers on {R}iemannian symmetric spaces
  of the noncompact type.
\newblock {\em Ann. of Math. (2) 132}, 3 (1990), 597--628.

\bibitem{BCCFR}
{\sc Betancor, J.~J., Castro, A.~J., Curbelo, J., Fari\~{n}a, J.~C., and
  Rodr\'{\i}guez-Mesa, L.}
\newblock Square functions in the {H}ermite setting for functions with values
  in {UMD} spaces.
\newblock {\em Ann. Mat. Pura Appl. (4) 193}, 5 (2014), 1397--1430.

\bibitem{BCCR}
{\sc Betancor, J.~J., Castro, A.~J., Curbelo, J., and Rodr\'{\i}guez-Mesa, L.}
\newblock Characterization of {UMD} {B}anach spaces by imaginary powers of
  {H}ermite and {L}aguerre operators.
\newblock {\em Complex Anal. Oper. Theory 7}, 4 (2013), 1019--1048.

\bibitem{BCN}
{\sc Betancor, J.~J., Castro, A.~J., and Nowak, A.}
\newblock Calder\'{o}n--{Z}ygmund operators in the {B}essel setting.
\newblock {\em Monatsh. Math. 167}, 3-4 (2012), 375--403.

\bibitem{BFRS}
{\sc Betancor, J.~J., Fari\~{n}a, J.~C., Rodr\'{\i}guez-Mesa, L., and
  Sanabria-Garc\'{\i}a, A.}
\newblock Higher order {R}iesz transforms for {L}aguerre expansions.
\newblock {\em Illinois J. Math. 55}, 1 (2011), 27--68 (2012).

\bibitem{BdL}
{\sc Betancor, J.~J., and León-Contreras, M.~D.}
\newblock Variation inequalities for {R}iesz transforms and {P}oisson
  semigroups associated with {L}aguerre polynomial expansions, 2021.
\newblock arXiv:2110.03493.

\bibitem{CLR}
{\sc Chen, Y., Levine, S., and Rao, M.}
\newblock Variable exponent, linear growth functionals in image restoration.
\newblock {\em SIAM J. Appl. Math. 66}, 4 (2006), 1383--1406.

\bibitem{CW}
{\sc Coifman, R.~R., and Weiss, G.}
\newblock {\em Analyse harmonique non-commutative sur certains espaces
  homog\`enes}.
\newblock Lecture Notes in Mathematics, Vol. 242. Springer-Verlag, Berlin-New
  York, 1971.
\newblock \'{E}tude de certaines int\'{e}grales singuli\`eres.

\bibitem{CUFMP}
{\sc Cruz-Uribe, D., Fiorenza, A., Martell, J.~M., and P\'{e}rez, C.}
\newblock The boundedness of classical operators on variable {$L^p$} spaces.
\newblock {\em Ann. Acad. Sci. Fenn. Math. 31}, 1 (2006), 239--264.

\bibitem{CUFN}
{\sc Cruz-Uribe, D., Fiorenza, A., and Neugebauer, C.~J.}
\newblock The maximal function on variable {$L^p$} spaces.
\newblock {\em Ann. Acad. Sci. Fenn. Math. 28}, 1 (2003), 223--238.

\bibitem{CUF}
{\sc Cruz-Uribe, D.~V., and Fiorenza, A.}
\newblock {\em Variable {L}ebesgue spaces}.
\newblock Applied and Numerical Harmonic Analysis. Birkh\"{a}user/Springer,
  Heidelberg, 2013.
\newblock Foundations and harmonic analysis.

\bibitem{DS1}
{\sc Dalmasso, E., and Scotto, R.}
\newblock Riesz transforms on variable {L}ebesgue spaces with {G}aussian
  measure.
\newblock {\em Integral Transforms Spec. Funct. 28}, 5 (2017), 403--420.

\bibitem{DS2}
{\sc Dalmasso, E., and Scotto, R.}
\newblock New {G}aussian {R}iesz transforms on variable {L}ebesgue spaces.
\newblock To appear in Anal. Math., 2021.

\bibitem{Die}
{\sc Diening, L.}
\newblock Maximal function on generalized {L}ebesgue spaces {$L^{p(\cdot)}$}.
\newblock {\em Math. Inequal. Appl. 7}, 2 (2004), 245--253.

\bibitem{DHHR}
{\sc Diening, L., Harjulehto, P., H\"{a}st\"{o}, P., and R{\r u}\v{z}i\v{c}ka,
  M.}
\newblock {\em Lebesgue and {S}obolev spaces with variable exponents},
  vol.~2017 of {\em Lecture Notes in Mathematics}.
\newblock Springer, Heidelberg, 2011.

\bibitem{DR}
{\sc Diening, L., and R{\r u}\v{z}i\v{c}ka, M.}
\newblock Calder\'{o}n-{Z}ygmund operators on generalized {L}ebesgue spaces
  {$L^{p(\cdot)}$} and problems related to fluid dynamics.
\newblock {\em J. Reine Angew. Math. 563\/} (2003), 197--220.

\bibitem{Di}
{\sc Dinger, U.}
\newblock Weak type {$(1,1)$} estimates of the maximal function for the
  {L}aguerre semigroup in finite dimensions.
\newblock {\em Rev. Mat. Iberoamericana 8}, 1 (1992), 93--120.

\bibitem{FSS1}
{\sc Forzani, L., Sasso, E., and Scotto, R.}
\newblock Weak-type inequalities for higher order {R}iesz-{L}aguerre
  transforms.
\newblock {\em J. Funct. Anal. 256}, 1 (2009), 258--274.

\bibitem{GMMST}
{\sc Garc\'{\i}a-Cuerva, J., Mauceri, G., Meda, S., Sj\"{o}gren, P., and
  Torrea, J.~L.}
\newblock Functional calculus for the {O}rnstein-{U}hlenbeck operator.
\newblock {\em J. Funct. Anal. 183}, 2 (2001), 413--450.

\bibitem{GLY}
{\sc Grafakos, L., Liu, L., and Yang, D.}
\newblock Vector-valued singular integrals and maximal functions on spaces of
  homogeneous type.
\newblock {\em Math. Scand. 104}, 2 (2009), 296--310.

\bibitem{Leb}
{\sc Lebedev, N.~N.}
\newblock {\em Special functions and their applications}.
\newblock Dover Publications, Inc., New York, 1972.
\newblock Revised edition, translated from the Russian and edited by Richard A.
  Silverman, Unabridged and corrected republication.

\bibitem{Lo}
{\sc Lorist, E.}
\newblock On pointwise {$\ell^r$}-sparse domination in a space of homogeneous
  type.
\newblock {\em J. Geom. Anal. 31}, 9 (2021), 9366--9405.

\bibitem{MPS}
{\sc Men\'{a}rguez, T., P\'{e}rez, S., and Soria, F.}
\newblock The {M}ehler maximal function: a geometric proof of the weak type 1.
\newblock {\em J. London Math. Soc. (2) 61}, 3 (2000), 846--856.

\bibitem{MRA}
{\sc Meskhi, A., Rafeiro, H., and Asad~Zaighum, M.}
\newblock Interpolation of an analytic family of operators on variable exponent
  {M}orrey spaces.
\newblock {\em Hiroshima Math. J. 48}, 3 (2018), 335--346.

\bibitem{MPU}
{\sc Moreno, J., Pineda, E., and Urbina, W.}
\newblock The {B}oundedness of the {O}rnstein-{U}hlenbeck semigroup on variable
  {L}ebesgue spaces with respect to the {G}aussian measure, 2019.
\newblock arXiv:1911.06375.

\bibitem{Mu1}
{\sc Muckenhoupt, B.}
\newblock Poisson integrals for {H}ermite and {L}aguerre expansions.
\newblock {\em Trans. Amer. Math. Soc. 139\/} (1969), 231--242.

\bibitem{Mu2}
{\sc Muckenhoupt, B.}
\newblock Conjugate functions for {L}aguerre expansions.
\newblock {\em Trans. Amer. Math. Soc. 147\/} (1970), 403--418.

\bibitem{NPU}
{\sc Navas, E., Pineda, E., and Urbina, W.}
\newblock The {B}oundedness of {G}eneral {A}lternative {G}aussian {S}ingular
  {I}ntegrals on variable {L}ebesgue spaces with {G}aussian measure, 2021.
\newblock arXiv:2006.12985.

\bibitem{N}
{\sc Nowak, A.}
\newblock On {R}iesz transforms for {L}aguerre expansions.
\newblock {\em J. Funct. Anal. 215}, 1 (2004), 217--240.

\bibitem{NS}
{\sc Nowak, A., and Stempak, K.}
\newblock {$L^2$}-theory of {R}iesz transforms for orthogonal expansions.
\newblock {\em J. Fourier Anal. Appl. 12}, 6 (2006), 675--711.

\bibitem{Pal}
{\sc Palka, B.~P.}
\newblock {\em An introduction to complex function theory}.
\newblock Undergraduate Texts in Mathematics. Springer-Verlag, New York, 1991.

\bibitem{P}
{\sc P\'{e}rez, S.}
\newblock The local part and the strong type for operators related to the
  {G}aussian measure.
\newblock {\em J. Geom. Anal. 11}, 3 (2001), 491--507.

\bibitem{RRT}
{\sc Rubio~de Francia, J.~L., Ruiz, F.~J., and Torrea, J.~L.}
\newblock Calder\'{o}n-{Z}ygmund theory for operator-valued kernels.
\newblock {\em Adv. in Math. 62}, 1 (1986), 7--48.

\bibitem{Ru}
{\sc R{\r u}\v{z}i\v{c}ka, M.}
\newblock {\em Electrorheological fluids: modeling and mathematical theory},
  vol.~1748 of {\em Lecture Notes in Mathematics}.
\newblock Springer-Verlag, Berlin, 2000.

\bibitem{Sa4}
{\sc Sasso, E.}
\newblock Spectral multipliers of {L}aplace transform type for the {L}aguerre
  operator.
\newblock {\em Bull. Austral. Math. Soc. 69}, 2 (2004), 255--266.

\bibitem{Sa2}
{\sc Sasso, E.}
\newblock Functional calculus for the {L}aguerre operator.
\newblock {\em Math. Z. 249}, 3 (2005), 683--711.

\bibitem{Sa3}
{\sc Sasso, E.}
\newblock Maximal operators for the holomorphic {L}aguerre semigroup.
\newblock {\em Math. Scand. 97}, 2 (2005), 235--265.

\bibitem{Sa1}
{\sc Sasso, E.}
\newblock Weak type estimates for {R}iesz-{L}aguerre transforms.
\newblock {\em Bull. Austral. Math. Soc. 75}, 3 (2007), 397--408.

\bibitem{StLP}
{\sc Stein, E.~M.}
\newblock {\em Topics in harmonic analysis related to the {L}ittlewood-{P}aley
  theory}.
\newblock Annals of Mathematics Studies, No. 63. Princeton University Press,
  Princeton, N.J.; University of Tokyo Press, Tokyo, 1970.

\bibitem{Sz}
{\sc Szeg\"{o}, G.}
\newblock {\em Orthogonal polynomials}.
\newblock American Mathematical Society Colloquium Publications, Vol. 23.
  American Mathematical Society, Providence, R.I., 1959.
\newblock Revised ed.

\bibitem{Yos}
{\sc Yosida, K.}
\newblock {\em Functional analysis}, sixth~ed., vol.~123 of {\em Grundlehren
  der Mathematischen Wissenschaften [Fundamental Principles of Mathematical
  Sciences]}.
\newblock Springer-Verlag, Berlin-New York, 1980.

\end{thebibliography}
\end{document}